\documentclass[12pt,reqno]{amsart}
\usepackage{amsfonts,amssymb,amsthm}
\usepackage{amsmath}
\usepackage[numbers,sort&compress]{natbib}
\usepackage[mathscr]{euscript}
\usepackage{graphicx}
\usepackage{booktabs}
\usepackage{color}
\usepackage{mathrsfs}
\usepackage{paralist}
\usepackage{hyperref}
\usepackage{framed}

\newtheorem {theorem}{Theorem}[section]
\newtheorem {proposition}{Proposition}[section]
\newtheorem {corollary}{Corollary}[section]

\newtheorem {lemma}{Lemma}[section]
\newtheorem{algorithm}[theorem]{Algorithm}
\newtheorem {example}{Example}[section]
\newtheorem {definition}{Definition}[section]
\newtheorem {remark}{Remark}[section]
\newtheorem {assumption}{Assumption}[section]
\newtheorem{condition}[theorem]{Condition}

\newcommand{\Int}[1]{{\sf int}(#1)}

\newcommand{\bfz}{0}
\newcommand{\bfV}{\mathbf{V}}

\newcommand{\RR}{\mathbb{R}}

\newcommand{\jac}{{\sf Jac}}
\newcommand{\CC}{{\mathbb C}}
\newcommand{\GAMMA}{\Gamma(f,S,x^*)}
\newcommand{\GAMMAP}{\Gamma_P(f,S,x^*)}
\newcommand{\GAMMAA}{\Gamma_{A^{-T}A^{-1}}(f,S,x^*)}

\def\ees{{\accent"5E e}\kern-.385em\raise.2ex\hbox{\char'23}\kern-.08em}
\def\EES{{\accent"5E E}\kern-.5em\raise.8ex\hbox{\char'23 }}
\def\ow{o\kern-.42em\raise.82ex\hbox{
   \vrule width .12em height .0ex depth .075ex \kern-0.16em \char'56}\kern-.07em}
\def\OW{O\kern-.460em\raise1.36ex\hbox{
\vrule width .13em height .0ex depth .075ex \kern-0.16em \char'56}\kern-.07em}

\title{On types of KKT points in polynomial optimization}

\author{Feng Guo}
\address[Feng Guo]{School of Mathematical Sciences, Dalian  University of Technology, Dalian, 116024, China}
\email{fguo@dlut.edu.cn}

\author{Do Sang Kim}
\address[Do Sang Kim]{Department of Applied Mathematics, Pukyong National University, Busan, 48513, Korea}
\email{dskim@pknu.ac.kr}

\author{Liguo Jiao$^*$}
\address[Liguo Jiao]{School of Mathematical Sciences, Soochow University, Suzhou 215006, Jiangsu Province, China}
\email{hanchezi@163.com}
\thanks{$^*$Corresponding Author.}

\author{TI\EES N-S\OW N PH\d{A}M}
\address[Ti\ees n-S\ow n Ph\d{a}m]{Department of Mathematics, University of Dalat, 1, Phu Dong Thien Vuong, Dalat, Vietnam}
\email{sonpt@dlu.edu.vn}
\date{\today}

\begin{document}

\begin{abstract}
Let $f$ be a real polynomial function with $n$ variables and $S$ be a
basic closed semialgebraic set in $\RR^n$.
In this paper, we are interested in the problem of identifying the type (local
minimizer, maximizer or not extremum point) of a given isolated KKT point
$x^*$ of $f$ over $S.$
To this end, we investigate some properties of
the tangency variety of $f$ on $S$ at $x^*,$ by which we introduce
the definition of {\it faithful radius} of $f$ over $S$ at $x^*.$
Then, we show that the type of $x^*$ can be determined by the global
extrema of $f$ over the intersection of $S$ and the Euclidean ball
centered at $x^*$ with a faithful radius.  Finally, we propose
an algorithm involving algebraic computations to compute a faithful
radius of $x^*$ and determine its type.
\end{abstract}
\subjclass[2010]{65K05, 68W30}
\keywords{polynomial functions, KKT points, tangency varieties, types,
faithful radii.}
\maketitle

\section{Introduction}\label{sec::intro}
Consider the following constrained polynomial optimization problem
\begin{equation}\label{eq::P}
	\left\{\begin{aligned}
		\min_{x\in\RR^n}&\ \ f(x)\\
		\text{s.t.}&\ \ g_1(x)=0, \ldots, g_l(x)=0,\\
		&\ \ h_1(x)\ge 0, \ldots, h_m(x)\ge 0,
	\end{aligned}\right.
\end{equation}
where $f(x)$, $g_i(x)$'s, $h_j(x)$'s $\in\RR[x]$ are polynomials in
$x=(x_1,\ldots,x_n)$ with real coefficients.
Denote by $S$ the feasible set of \eqref{eq::P} which is a {\itshape basic closed
semialgebraic set} in $\RR^n$.

Let $x^*\in S$ be a Karush--Kuhn--Tucker (KKT for short) point of \eqref{eq::P}, i.e., the first order necessary optimality
conditions
\[
\nabla f(x^*)-\sum_{i = 1}^{l}\lambda_i \nabla g_i(x^*)-\sum_{j\in
J(x^*)}\nu_j \nabla h_j(x^*) = 0 \quad \text{and}\quad  \nu_j\ge 0,\ \ j\in J(x^*),
\]
hold at $x^*$ for some Lagrange multipliers $\lambda_i$'s, $\nu_j$'s $\in
{\Bbb R}$, where $J(x^*)$ denotes the active set at $x^*$.
Our goal in this paper is
to determine the type of $x^*$. In other words, is $x^*$ a local minimizer,
maximizer or not extremum point of \eqref{eq::P}?

The motivation of our research is as follows. Many nonlinear
programming algorithms are designed to generate a sequence of points
which, under certain conditions, converges to a KKT point. However,
there is no theoretical guarantee that the obtained KKT point is local
minimizer. Some algorithms are purely based on solving the system of
first order optimality conditions of the optimization problem. Hence,
the obtained KKT point may even be a maximizer. Unfortunately, we will
see that testing the type of a KKT point may be a hard problem as
shown by Murty and Kabadi in \cite{Murty1987}. In fact, for
\eqref{eq::P}, to decide the type of $x^*$, we may first consider the
second-order necessary condition for $x^*$ to be a local minimizer
(resp., maximizer). Let $H$ be the Hessian matrix of the Lagrange function
of \eqref{eq::P} with respect to $x$ at $x^*$. We need to check
whether $y^T H y$ is nonnegative (resp., nonpositive) for all $y \in Y,$ where
\[
Y:=\left\{y\in\RR^n :
		\left\{
		\begin{aligned}
			&\nabla g_i(x^*)^Ty=0,\ i=1,\ldots,l,\\
			&\nabla h_j(x^*)^Ty=0,\ \text{for all}\ j\in J(x^*)\
			\text{with}\ \nu_j>0\\
			&\nabla h_j(x^*)^Ty\ge 0,\ \text{for all}\ j\in J(x^*)\
			\text{with}\ \nu_j=0\\
		\end{aligned}
		\right.
\right\}.
\]
If the second-order necessary condition holds, we may further consider
the second-order
sufficient condition. That is, to decide if $y^T H y$ is positive
(resp., negative) for all $0\neq y\in Y$.
However, when $H$ is not positive semidefinite and the set $\{j\mid
j\in J(x^*), \ \nu_j=0\}$ is nonempty, Murty and Kabadi showed that if
the entries in $H$, $\nabla g_i(x^*)$'s, $\nabla h_j(x^*)$'s are rational, then
checking whether the second-order sufficient condition holds is
co-NP-complete (c.f. \cite[Theorem 4]{Murty1987}). Even if we are able
to check the  second-order sufficient condition, if it is not
satisfied, there is no straightforward and simple method to determine
whether $x^*$ is a local minimizer (resp., maximizer) by present theory. In
particular, Murty and Kabadi showed (c.f. \cite[Theorem 2]{Murty1987}) that checking if the KKT point
$0$ is not a local minimizer for the nonconvex quadratic problem
$\min_{x\ge 0} x^TDx$ is NP-complete, where $D$ is a not positive matrix.

To the best of our knowledge, there is very little related work in the
literature addressing this issue, even in the unconstrained case. If
there is no constraint, the problem reduces to determining the type
of a degenerate real critical point of $f$, i.e., a point at which the
gradient $\nabla f$ vanishes and the Hessian matrix $\nabla^2f$ is
singular. To solve it, it is
intuitive to consider the higher order partial derivatives of $f$.
However, it is difficult to take into account only the higher order
derivatives of $f,$ to systematically solve this problem.
When $f$ is a sufficiently smooth function (not necessarily a
polynomial), some partial answers to this problem were given in
\cite{Bolis1980,Cushing1975} under certain assumptions on its Taylor
expansion at the point. When $f$ is a multivariate real polynomial, Qi
investigated its critical points and extrema structures in
\cite{Qi2004} without giving a method to determine their
types.  Nie gave a numerical method in \cite{Nie2015} to compute all
$H$-minimizers (critical points at which the Hessian matrices are
positive semidefinite) of a polynomial by semidefinite relaxations.
However, there is no completed procedure in \cite{Nie2015} to verify
that a $H$-minimizer is a saddle point.

Very recently, Guo and Ph\d{a}m~\cite{GP2017} proposed a method to
determine the type of an isolated degenerate real critical point of a
multivariable real polynomial. They showed that the type of the
critical point can be determined by the global extrema of the
polynomial over the Euclidean ball centered at the critical point with
the so-called {\itshape faithful radius}. An algorithm involving algebraic
computations to compute a faithful radius of the critical point is
given in \cite{GP2017}. To decide the type, instead of computing the
extrema of the polynomial over the ball which itself is NP-hard, they
presented an algorithm to identify the type by computing isolating
intervals for each real root of a zero-dimensional polynomial system,
which can be carried out efficiently (c.f.
\cite{ABRW,FGLM1993,GHMP,GIUSTI2001154,Rouillier1999}).
In this paper, we extend the method in \cite{GP2017} to constrained
case \eqref{eq::P}. We generalize the definition of faithful radius to
an isolated KKT point $x^*$ by means of the tangency variety of $f$ at
$x^*$ over the constraint $S$ and derive analogue strategies as proposed in \cite{GP2017} to
decide the type of $x^*$.

To end this section, we would like to point out that determining the type of
the KKT point $x^*$ is a special case of the {\itshape quantifier
elimination problem}. Precisely,
determining the type of $x^*$ is equivalent to
checking the truth of the following first-order sentences
\begin{equation}\label{eq::cad}
	\begin{aligned}
		\forall r \exists x\ \
		(r=0)\vee((\Vert x-x^*\Vert^2\le
		&r^2)\wedge(g_1(x)=0)\wedge\cdots\wedge(g_l(x)=0)&\\
	&\wedge(h_1(x)\ge 0)\wedge\cdots\wedge(h_m(x)\ge 0)\wedge(f(x)>f(x^*))),&\\
		\forall r \exists x\ \
		(r=0)\vee((\Vert x-x^*\Vert^2\le
		&r^2)\wedge(g_1(x)=0)\wedge\cdots\wedge(g_l(x)=0)&\\
	&\wedge(h_1(x)\ge 0)\wedge\cdots\wedge(h_m(x)\ge 0)\wedge(f(x)<f(x^*))),&\\
	\end{aligned}
\end{equation}
where $\vee$ and $\wedge$ respectively denote the logical connectives
``or'' and ``and''. These decision problems can be solved by
algorithms based on the {\itshape cylindrical algebraic decomposition}
(CAD) \cite{ARG,QECAD}.
However, the arithmetic complexity for solving them by the
CAD is ${((l+m+3)D)^{O(1)}}^{n+1}$ where $D\ge 2$ is a bound for the
degrees of $f$, $g_i$'s and $h_j$'s  \cite[Excercise 11.7]{ARG}. The
complexity
is doubly exponential in $n$ and limits its practical application to
nontrivial problems. Indeed, a cylindrical decomposition of the whole
space seems to be superfluous for determining the local extremality of
$x^*$. Comparatively, by investigating the local values of $f$ on its
tangency variety on $S$ at $x^*$ (Definition \ref{def::tangency}), the
approach proposed in this paper
enjoys a lower complexity at least in certain circumstances (see
discussions in Section \ref{sec::complexity}), which can be observed
from the numerical experiments in Section \ref{sec::computation}.

\vskip 5pt
The paper is organized as follows. Some notation and preliminaries
used in this paper are given in Section~\ref{SecPreliminaries}.
We study some properties of the set of KKT points and tangency
varieties in Section~\ref{SectionTangencyVariety}.
In Section~\ref{sec::type}, we define the faithful radius of an
isolated KKT point, by which we show how to decide the type of the KKT
point. Some computational aspects are investigated in
Section~\ref{sec::computation},  where the algorithm for
determining its type of an isolated KKT point are presented. For a
better readabilty, we put the correctness proof and complexity
discussions of the algorithm in the Appendix \ref{appendix}.

\section{Preliminaries} \label{SecPreliminaries}

We use the following notation and terminology.
The symbol $\RR$ (resp., $\CC$) denotes the set of real (resp., complex) numbers.
We denote by $\RR_+$ the set of nonnegative real numbers.
$\RR[x]=\RR[x_1,\ldots,x_n]$ denotes the ring of polynomials in variables
$x=(x_1,\ldots,x_n)$ with real coefficients. The Euclidean space $\mathbb{R}^n$ is equipped with the usual scalar
product $\langle \cdot, \cdot \rangle$ and the corresponding Euclidean
norm $\|\cdot\|.$ For convenience, let $\Vert
x\Vert^2:=x_1^2+\cdots+x_n^2$ for any $x\in\CC^n$.
Denote $\RR^{n\times n}$ (resp., $\CC^{n\times n}$) as the set of $n\times n$ matrices
with real (resp., complex) number entries.
Denote by $\Vert A\Vert$ the $2$-norm of a matrix $A\in\RR^{n\times n}$.
For $R > 0$, denote by $\mathbb{B}_R(x)$ (resp., $\mathbb{B}_R$) the
closed ball with center $x$ (resp., $0$) and radius $R.$ For a subset $S \subset \mathbb{R}^n$, the interior and closure of $S$ in Euclidean topology is denoted by
${\Int S}$ and $\bar{S},$ respectively. The notation $C^p$ means $p$-times continuously differentiable;
$C^\infty$ is infinitely continuously differentiable.
If $f, g$ are two functions with suitably chosen domains and codomains, then
$f\circ g$ denotes the composite function of $f$ and $g$.

\subsection{Semialgebraic geometry}
Let us recall some notion and results from semialgebraic geometry (see,
for example, \cite{Benedetti1990, Bochnak1998, Dries1996}) which we
need.

\begin{definition}{\rm
\begin{enumerate}
  \item[(i)] A subset of $\mathbb{R}^n$ is said {\em semialgebraic} if
	  it is a finite union of sets of the form
$$\{x \in \mathbb{R}^n \ | \ f_i(x) = 0, i = 1, \ldots, k;\ f_i(x) > 0,
i = k + 1, \ldots, p\},$$
where all $f_{i}$'s are in $\RR[x]$.
 \item[(ii)]
Let $A \subset \Bbb{R}^n$ and $B \subset \Bbb{R}^m$ be semialgebraic
sets. A map $F \colon A \to B$ is said to be {\em semialgebraic} if
its graph
$$\{(x, y) \in A \times B \ | \ y = F(x)\}$$
is a semialgebraic subset in $\Bbb{R}^n\times\Bbb{R}^m.$
\end{enumerate}
}\end{definition}

The class of semialgebraic sets is closed under taking finite
intersections, finite unions, and complements; a Cartesian product of
semialgebraic sets is a semialgebraic set. Moreover, a major fact
concerning the class of semialgebraic sets is its stability under
linear projections (see, for example,~\cite{Benedetti1990,Bochnak1998}).

\begin{theorem}[Tarski--Seidenberg Theorem] \label{TarskiSeidenbergTheorem}
The image of a semialgebraic set by a semialgebraic map is semialgebraic.
\end{theorem}

By the Tarski--Seidenberg Theorem, it is not hard to see that the
closure and the interior of a semialgebraic set are semialgebraic
sets.

Recall the Curve Selection Lemma which will be used in this paper (see, for example, \cite{Ha2017,Milnor1968}).

\begin{lemma}[Curve Selection Lemma]\label{CurveSelectionLemma}
Let $A$ be a semialgebraic subset of $\mathbb{R}^n,$ and $u^* \in
\overline{A}\setminus A.$ Then there exists a real analytic
semialgebraic curve
$$\phi \colon (-\epsilon, \epsilon) \to {\mathbb R}^n$$
with $\phi(0) = u^*$ and with $\phi(t) \in A$ for $t \in (0, \epsilon).$
\end{lemma}

In what follows, we will need the following useful results (see, for example, \cite{Dries1996}).

\begin{lemma}[Monotonicity Lemma] \label{MonotonicityLemma}
Let $a < b$ in $\mathbb{R}.$ If $f \colon [a, b] \rightarrow
\mathbb{R}$ is a semialgebraic function, then there is a partition $a
=: t_1 < \cdots < t_{N} := b$ of $[a, b]$ such that $f|_{(t_l, t_{l +
1})}$ is $C^1,$ and either constant or strictly monotone$,$ for $l \in
\{1, \ldots, N - 1\}.$
\end{lemma}

The next theorem (see \cite{Bochnak1998,Dries1996}) uses the concept
of a cell whose definition we omit. We do not need the specific
structure of cells described in the formal definition. For us, it will
be sufficient to think of a $C^p$-cell of dimension $r$ as of an
$r$-dimensional $C^p$-manifold, which is the image of the cube $(0,
1)^r$ under a semialgebraic $C^p$-diffeomorphism. As follows from the
definition, an $n$-dimensional cell in $\mathbb{R}^n$ is an open set.

\begin{theorem}[Cell Decomposition Theorem] \label{CellTheorem}
Let $A \subset \mathbb{R}^n$ be a  semialgebraic set. Then$,$ for any
$p \in \mathbb{N},$ $A$ can be represented as a disjoint union of a
finite number of cells of class~$C^p.$
\end{theorem}

By Cell Decomposition Theorem, for any $p \in \mathbb{N}$ and any
nonempty semialgebraic subset $A$ of $\mathbb{R}^n,$ we can write $A$
as a disjoint union of finitely many semialgebraic $C^p$-manifolds of
different dimensions. The {\em dimension} $\dim A$ of a nonempty
semialgebraic set $A$ can thus be defined as the dimension of the
manifold of highest dimension of its decomposition. This dimension is
well defined and independent of the decomposition of $A.$ By
convention, the dimension of the empty set is taken to be negative
infinity. We will need the following result (see
\cite{Bochnak1998,Dries1996}).

\begin{proposition} \label{DimensionProposition}
\begin{enumerate}
\item [{\rm (i)}] Let $A \subset \mathbb{R}^n$ be a semialgebraic set and
$f \colon A \to\mathbb{R}^m$ a semialgebraic map. Then$,$ $\dim f(A)\leq\dim  A.$

\item [{\rm (ii)}] Let $A \subset \mathbb{R}^n$ be a nonempty
	semialgebraic set. Then$,$ $\dim(\bar{A} \setminus A) < \dim A.$
	In particular$,$ $\dim(\bar{A})=\dim A.$

\item [{\rm (iii)}] Let $A, B \subset \mathbb{R}^n$ be semialgebraic sets. Then$,$
 $$\dim (A \cup B) = \max\{ \dim A, \dim B\}.$$
\end{enumerate}
\end{proposition}

Combining Theorems~2.4.4,~2.4.5 and Proposition~2.5.13 in \cite{Bochnak1998}, it follows that
\begin{proposition}\label{prop1}
Let $A$ be a semialgebraic set of $\Bbb R^n.$ The following statements hold.
\begin{enumerate}
    \item [{\rm (i)}] $A$ has a finite number of connected components which are closed in $A.$
    \item [{\rm (ii)}] $A$ is connected if and only if it is path connected.
\end{enumerate}
\end{proposition}
Hence, in the rest of this paper, by saying that a semialgebraic subset of $\RR^n$
is connected, we also mean that it is path connected.

Next we state a semialgebraic version of Sard's theorem with the
parameter  in a simplified form which is sufficient for the
applications studied here. Given a differentiable map between
manifolds $f \colon X \rightarrow Y,$ a point $y \in Y$ is called a
{\em regular value} \index{regular value} for $f$ if either $f^{-1}(y)
= \emptyset$ or the derivative map $Df(x) \colon T_xX \rightarrow
T_yY$ is surjective at every point $x \in f^{-1}(y),$ where $T_x X$
and $T_yY$ denote the tangent spaces of $X$ at $x$ and of $Y$ at $y,$
respectively. A point $y \in Y$ that is not a regular value of $f$ is
called a {\em critical value.} The following result is also called Thom's weak transversality theorem.

\begin{theorem}[Sard's theorem with parameter] \label{SardTheorem}
Let $f \colon X \times \mathscr{P} \rightarrow Y$ be a differentiable
semialgebraic map between semialgebraic submanifolds.
If $y \in Y$ is a regular value of $f,$ then there exists a
semialgebraic set $\Sigma \subset \mathscr{P}$
of dimension smaller than the dimension of $\mathscr{P}$ such that$,$ for every $p
\in \mathscr{P} \setminus \Sigma,$ $y$ is a regular value of the map
$f_p \colon X \rightarrow Y, x \mapsto f(x, p).$
\end{theorem}
\begin{proof}
For a proof, we refer the reader to~\cite{DT} or \cite[Theorem~1.10]{Ha2017}.
\end{proof}

\subsection{Algebraic geometry}
A subset $I\subseteq\RR[x]$ is said an ideal if $0\in I$, $I+I\subseteq I$ and
$p\cdot q\in I$ for all $p\in I$ and $q\in\RR[x]$.
For $g_1,\ldots,g_s\in\RR[x]$,
denote $\langle g_1,\ldots,g_s\rangle$ as the ideal in $\RR[x]$ generated by $g_i$'s,
i.e., the set $g_1\RR[x]+\cdots+g_s\RR[x]$.
An ideal is radical if $f^m\in I$ for some integer $m\ge 1$ implies that $f\in I$.
The radical of an ideal $I\subseteq\RR[x]$, denoted $\sqrt{I}$, is the set
$\{f\in\RR[x]\mid f^m\in I\text{ for some integer }m\ge 1\}$.
An affine variety (resp., real affine variety) is a subset of $\CC^n$
(resp., $\RR^n$) that consists of common
zeros of a set of polynomials.
For an ideal $I\subseteq\RR[x]$, denote $\bfV_\CC(I)$ and $\bfV_\RR(I)$ as the affine
varieties defined by $I$ in $\CC^n$ and $\RR^n$, respectively.
For a polynomial $g\in\RR[x]$,
respectively replace $\bfV_\CC(\langle g\rangle)$ and $\bfV_\RR(\langle g\rangle)$ by
$\bfV_\CC(g)$ and $\bfV_\RR(g)$ for simplicity.
Given a set
$V\subseteq\CC^n$, denote $\mathbf{I}(V)\subseteq\RR[x]$ as the vanishing ideal of
$V$ in $\RR[x]$, i.e., the set of all polynomials in $\RR[x]$ which equal zero at
every point in $V$.
For an ideal $I\subseteq\RR[x]$, denote $\dim(I)$ as the Hilbert
dimension of $I$, i.e., the degree of the affine Hilbert polynomial of $I$.
For an ideal $I\subseteq\RR[x]$,
the decomposition $I=I_1\cap\cdots\cap I_s$ is said the equidimensional decomposition
of $I$ if each ideal $I_i$ is pure dimensional, i.e., all its associated
primes have the same dimension.
For an affine variety $V\subseteq\CC^n$,
denote $\dim(V)=\dim(\mathbf{I}(V))$ as its dimension.
When $\bfV_\CC(I)$ is finite, the ideal $I$ is said to be zero-dimensional.
For any subset $S\subseteq\CC^n$, denote $\overline{S}^\mathcal{Z}$ as the Zariski
closure of $S$ in $\CC^n$, i.e., $\overline{S}^{\mathcal{Z}}=\bfV_\CC(\mathbf{I}(S))$.

\section{KKT points and tangencies} \label{SectionTangencyVariety}

Recall the polynomials $f$, $g_i$'s, $h_j$'s $\in\RR[x]$ in \eqref{eq::P} and the basic closed semialgebraic set
$$S=\{x \in {\Bbb R}^n \ | \ g_i(x) = 0,\ i = 1, \ldots, l,\  h_j(x)
\ge 0,\ j = 1, \ldots, m \}.$$
Let $x^*\in S$ be a fixed KKT point in the rest of this paper. We also
assume that $x^*$ is not an isolated point of $S.$

\begin{definition}{\rm
\begin{enumerate}
\item[(i)] The point $x^*$ is said to be a {\em local minimizer} of
	$f$ on $S$ if there is an open neighborhood $U$ of $x^*$ such
	that
\begin{align*}
f(x^*) \le f(x) \quad \textrm{ for all } \ x \in S \cap U.
\end{align*}
\item[(ii)] The point $x^*$ is said to be a {\em local maximizer} of
	$f$ on $S$ if there is an open neighborhood $U$ of $x^*$ such
	that
\begin{align*}
f(x^*) \ge f(x) \quad \textrm{ for all } \ x \in S \cap U.
\end{align*}
\item[(iii)] The point $x^*$ is not an {\em extremum point} of $f$ on $S$ if for any open
		neighborhood $U$ of $x^*$, there exist $u , v \in S \cap U$
		such that
\begin{align*}
f(u) < f(x^*) < f(v).
\end{align*}
\end{enumerate}
}\end{definition}

\subsection{KKT points}
As is well known,
most numerical optimization methods targeting local (including global) minimizers are
often based on the following KKT optimality conditions:
\begin{eqnarray*}
&& \nabla  f(x)  - \sum_{i = 1}^{l} \lambda_i \nabla g_i(x) - \sum_{j = 1}^{m} \nu_j \nabla h_j(x) = 0,  \\
&& g_i(x) = 0,\ i = 1, \ldots, l, \ h_j(x) \ge 0,\ j = 1, \ldots, m,\\
&& \nu_j h_j(x) = 0,\ \nu_j \ge 0,\ j = 1, \ldots, m,
\end{eqnarray*}
where the variables $\lambda_i, \nu_j \in {\Bbb R}$ are said to be
Lagrange multipliers and $\nabla f$ denotes the vector whose
components are the partial derivatives of $f.$

Sometimes the above KKT system fails to hold at some minimizers. Hence, we
usually make an assumption said a {\em constraint qualification} to
ensure that the KKT system holds. Such a constraint
qualification--probably the one most often used in the design of
algorithms--is defined as follows:

\begin{definition}{\rm
We say that the {\em linearly independent constraint qualification} ({\rm
(LICQ)} for short) holds at $x \in S$ if the system of the vectors
$\nabla g_i(x)$, $i = 1, \ldots, l$, $\nabla h_j(x)$, $j \in J(x),$ is
linearly independent, where $J(x)$ is the set of indices $j$ for which
$h_j$ vanishes at $x.$}
\end{definition}

\begin{remark}\label{rk::licq}{\rm
\begin{enumerate}
\item[(i)] Note that {\rm (LICQ)} is generally satisfied, for a proof see \cite[Theorem 6.1]{Ha2017}.

\item[(ii)] Since $x^*$ is not isolated in $S,$ we can see that if
	{\rm (LICQ)} holds at $x^*,$ then $n - l - \#J(x^*) \ge 1$ where
	$\#J(x^*)$ denotes the number of elements in $J(x^*)$ and so
	$n - l \ge 1.$
\end{enumerate}
}\end{remark}

\begin{lemma}\label{regular1}
If {\rm (LICQ)} holds at $x^*\in S,$ then there exists a real number
$R>0$ such that {\rm (LICQ)} holds at every $x\in S\cap \Bbb B_R (x^*).$
\end{lemma}
\begin{proof}
Since ${\rm (LICQ)}$ holds at $x^*,$ then the system of the vectors
$\nabla g_i(x^*),$ $i = 1, \ldots, l,$ $\nabla h_j(x^*),$ $j \in
J(x^*)$ is linearly independent. By continuity, for all $x$
near to $x^*,$ $J(x)\subset J(x^*)$ and the system of the vectors
$\nabla g_i(x),$ $i = 1, \ldots, l,$ $\nabla h_j(x),$ $j \in J(x),$ is
linearly independent.
\end{proof}

The following lemma says that if {\rm (LICQ)} holds at $x^* \in S,$
then the set $S$ intersects transversally the sphere $\{ x \in \Bbb
R^n \ | \ \| x - x^* \| = R\}$ for all $R > 0$ small enough.
\begin{lemma}\label{includex}
If {\rm (LICQ)} holds at $x^* \in S,$ then there exists a real number
$R > 0$ such that the vectors $\nabla g_i(x), i = 1,
\ldots, l, \nabla h_j(x), j \in J(x),$ and $x - x^*$ are linearly independent for all $x \in S \cap \Bbb B_R (x^*) \setminus \{x^*\}.$
\end{lemma}
\begin{proof}
Without loss of generality, assume $x^* = 0.$
Suppose that the lemma is not true, then there exists a sequence
$\{x^k\} \subset S$ tending to $0$ such that $x^k\neq 0$ and the system of the
vectors $\nabla g_i(x^k),$ $i = 1, \ldots, l,$ $\nabla h_j(x^k),$ $j \in
J(x^k),$ and $x^k$ is linearly dependent for all $k,$ i.e., there exist
$\lambda_i^k,$ $\nu_j^k$ and $\mu^k$ $\in \Bbb R$ such that
\begin{eqnarray*}
&& \sum_{i = 1}^{l}\lambda_i^k \nabla g_i(x^k) + \sum_{j \in J(x^k)}
\nu_j^k \nabla h_j(x^k) + \mu^k x^k = 0, \ \textrm{ and }\\
&& \sum_{i = 1}^{l} (\lambda_i^k)^2 + \sum_{j \in J(x^k)}( \nu_j^k)^2
+ (\mu^k)^2 = 1.
\end{eqnarray*}
By passing to a subsequence, if necessary, we may assume that
$J(x^k)=J\subset \{1,2,\ldots, m\}$ for all $k,$ and there exist the
following limits
$$\lambda_i^* := \lim_{k\rightarrow\infty}\lambda_i^k, \quad \nu_j^* := \lim_{k\rightarrow\infty}\nu_j^k, \quad \textrm{ and } \quad \mu^* := \lim_{k\rightarrow\infty}\mu^k.$$
Let
\begin{eqnarray*}
A:=\{(x, \lambda, \nu, \mu)\in \Bbb R^n\times\Bbb R^l\times\Bbb R^{\#
J}\times\Bbb R & | & g_i(x)=0,\ i=1, \ldots, l,\ h_j(x)=0,\ j\in J, \\
&&\sum_{i = 1}^{l}\lambda_i \nabla g_i(x ) + \sum_{j \in J} \nu_j
\nabla h_j(x) + \mu x = 0, \\
&&\sum_{i = 1}^{l}\lambda_i^2 + \sum_{j \in J} \nu_j^2 + \mu^2 =
1\}.
\end{eqnarray*}
Then $A$ is a semialgebraic set and $(0, \lambda^*, \nu^*, \mu^*)$ is a limit point of the
set $\{(x, \lambda, \nu, \mu)\in A\mid x\neq 0\}$ which is also
semialgebraic.
Using the Curve Selection Lemma~\ref{CurveSelectionLemma},
there exist a smooth semialgebraic curve $\varphi(t)$ and
semialgebraic functions $\lambda_i(t),$ $\nu_j(t),$ $\mu(t),$ $t \in
(-\epsilon, \epsilon),$ such that
\begin{enumerate}
\item[(a1)] $\big(\varphi(t), \lambda(t), \nu(t), \mu(t)\big)\in A$
	and $\varphi(t)\neq 0$ for $t \in (0, \epsilon);$
\item[(a2)] $\varphi(t) \rightarrow 0$ as $t \rightarrow 0^+.$
\end{enumerate}
It follows from (a1) that
\begin{eqnarray*}
0  &=& \sum_{i = 1}^{l} \lambda_i(t) \langle \nabla  g_i(\varphi(t)), \frac{d
\varphi(t)}{dt} \rangle + \sum_{j \in J} \nu_j(t) \langle \nabla
h_j(\varphi(t)), \frac{d \varphi(t)}{dt}
\rangle +  \mu (t) \langle \varphi(t), \frac{d \varphi(t)}{dt} \rangle \\
&=& \sum_{i = 1}^{l}  \lambda_i(t) \frac{d}{dt}(g_i \circ \varphi)(t) + \sum_{j \in
J} \nu_j(t) \frac{d}{dt}(h_j \circ \varphi)(t) +
\frac{\mu(t)}{2} \frac{d \|\varphi(t)\|^2}{dt} \\
&=& \frac{\mu (t)}{2} \frac{d \|\varphi(t)\|^2}{dt}
\end{eqnarray*}
holds for each $t\in(0,\epsilon)$. Applying the Monotonicity Lemma~\ref{MonotonicityLemma} and shrinking $\epsilon$ (if necessary), we may assume that the functions $\mu(t)$ and $\|\varphi(t)\|$ are either constant or strictly monotone. Then, (a1) implies that for each $t\in(0,\epsilon)$, $\mu(t)=0$ and hence the vectors $\nabla
g_i(\varphi(t)),$ $i = 1 , \ldots, l,$ $\nabla h_j(\varphi(t)),$  $j
\in J\subseteq J(\varphi(t)),$ are linearly dependent. By (a2), it
contradicts Lemma~\ref{regular1}.
\end{proof}

\begin{definition}{\rm
The set of {\em KKT points} of $f$ on $S$ is defined as follows:
\begin{eqnarray*}
\Sigma(f, S) := \{x \in S & \mid & \textrm{there exist $\lambda_i ,
	\nu_j \in {\Bbb R} $  such that } \\
&& \nabla  f(x)  - \sum_{i = 1}^{l} \lambda_i \nabla g_i(x) -
\sum_{j = 1}^{m} \nu_j \nabla h_j(x) = 0,  \ \textrm{ and } \\
&&  \nu_j h_j(x) = 0,\ j = 1, \ldots, m  \}.
\end{eqnarray*}}
\end{definition}

\begin{remark}{\rm
By the Tarski--Seidenberg Theorem~\ref{TarskiSeidenbergTheorem}, $\Sigma (f, S)$ is a
semialgebraic (possibly empty) set and so it has a finite number of
connected components. Moreover it is not hard to see that if
{\rm (LICQ)} holds at every point in $S$ then $f(\Sigma(f, S))$ is a finite set (see, for example, \cite[Theorem 2.3]{Ha2017}).
}\end{remark}

The following statement is well known (see, for example, \cite{Bertsekas2016}).

\begin{theorem}[KKT necessary optimality conditions] \label{th::KKT}
Assume that {\rm (LICQ)} holds at $x^* \in S.$ If $x^*$ is a local minimizer $($or maximizer$)$ of $f$ on $S,$ then $x^* \in \Sigma (f , S).$
\end{theorem}

\begin{corollary}\label{cor::notconstant}
	Assume that {\rm (LICQ)} holds at $x^* \in S$ and $x^*$ is an
	isolated KKT point of $f$ on $S.$ Then the restriction of $f$ on
	$S$ is nonconstant in some neighborhood of $x^*$.
\end{corollary}
\begin{proof}
	This follows  immediately from Lemma \ref{regular1} and Theorem \ref{th::KKT}.
\end{proof}

By the above corollary, we can see that if (LICQ) holds at $x^* \in S$ and $x^*$ is an isolated KKT point of $f,$
then $x^*$ is a local minimizer (resp., maximizer) of $f$ if and only if it is an isolated local minimizer (resp., maximizer) of $f.$

\subsection{Tangencies}

\begin{definition}{\rm \cite{HaHV2009-1}\label{def::tangency}
The {\em tangency variety} of $f$ on $S$ at $x^*$ is defined as follows:
\begin{eqnarray*}
\GAMMA:=  \{x \in S & \ | \ &  \textrm{there exist real numbers
$\kappa, \lambda_i, \nu_j, \mu,$ not all zero, such that } \\
&& \kappa \nabla  f(x)  - \sum_{i = 1}^{l} \lambda_i \nabla g_i(x) -
\sum_{j = 1}^{m} \nu_j \nabla h_j(x) - \mu (x - x^*) = 0,  \ \textrm{
and } \\
&&  \nu_j h_j(x) = 0,\ j = 1, \ldots, m  \}.
\end{eqnarray*}
}\end{definition}

\begin{lemma}\label{property}
The following statements hold$:$
\begin{enumerate}
\item[{\rm (i)}] $\Sigma(f, S)\subset\GAMMA;$
\item[{\rm (ii)}] $\GAMMA$ is a nonempty$,$ closed and semialgebraic set$;$ in
	particular$,$ it has a finite number of connected components$;$
\end{enumerate}
If the restriction of $f$ on $S$ is nonconstant in some neighborhood of $x^*,$ then
\begin{enumerate}
\item[{\rm (iii)}] $x^* \in \GAMMA$ and it is a limit point of
	$\GAMMA \setminus \Sigma (f, S);$
\item[{\rm (iv)}] For any $R > 0,$ $\dim \big( \GAMMA \setminus \Sigma
	(f, S) \big) \cap \mathbb{B}_R(x^*) \ge 1.$
\end{enumerate}
\end{lemma}
\begin{proof}
Without loss of generality, we assume $x^* = 0.$

(i) This is clear by definition.

(ii) For each $t \ge 0,$ let
$$S_t := S \cap \{x\in\RR^n\ | \ \|x\|=t\}.$$
Since $x^*$ is not an isolated point of $S,$ there exists $\epsilon > 0$ such that $S_t$ is a nonempty and compact set
for all $t \in [0, \epsilon).$
The set $\GAMMA$ is nonempty because it contains all
extremal points of $f$ on $S_t$ for all $t\in[0,\epsilon)$ by the
	Fritz-John necessary optimality conditions (see, for example,
	\cite{Bertsekas2016}). The closedness of $\GAMMA$ follows
	immediately from the
definition. By the Tarski--Seidenberg Theorem~\ref{TarskiSeidenbergTheorem}, $\GAMMA$ is a
semialgebraic set, and so it has a finite number of connected
components (due to Proposition~\ref{prop1}).

(iii) It is clear that $x^* \in \Gamma(f, S, x^*)$ by definition.
For the real number $\epsilon>0$ defined above, we can find two  semialgebraic curves
	$\varphi, \psi \colon [0 , \epsilon) \rightarrow \Bbb R^n$ such that
\begin{enumerate}
\item[(b1)] $\varphi(t)$ and $\psi(t)$ are the minimizer and maximizer of
	$f$ on $S_t$ for $t \in [0, \epsilon),$ respectively;
\item[(b2)] $\| \varphi(t) \| = \| \psi(t) \| = t \rightarrow 0$ as $t
	\rightarrow 0^+.$
\end{enumerate}
By the Fritz-John necessary optimality conditions (see, for example, \cite{Bertsekas2016}),
$\varphi (t), \psi(t) \in \GAMMA.$ Hence, $x^*$ is not isolated in $\GAMMA.$ By the Monotonicity Lemma~\ref{MonotonicityLemma}, we may assume that $\varphi$
and $\psi$ are differentiable on $(0, \epsilon)$ (perhaps after
reducing $\epsilon$).

Suppose that $\varphi(t) \in \Sigma(f , S)$ for all $t\in(0,\epsilon).$ Then there exist semialgebraic functions $\lambda_i , \nu_j \colon (0,
	\epsilon) \rightarrow {\Bbb R}$  such that
\begin{enumerate}
\item[(b3)] $\nabla  f(\varphi(t))  - \sum_{i = 1}^{l} \lambda_i(t)
	\nabla g_i(\varphi(t)) - \sum_{j = 1}^{m} \nu_j(t) \nabla
	h_j(\varphi(t)) \equiv 0.$
\item[(b4)] $\nu_j(t) h_j(\varphi(t)) \equiv 0,\ \ j = 1, \ldots, m.$
\end{enumerate}
Since the functions $\nu_j$ and $h_j \circ \varphi$ are semialgebraic,
for $\epsilon > 0$ small enough, these functions are either constant
or strictly monotone (thanks to the Monotonicity Lemma~\ref{MonotonicityLemma}).  Then, by (b4), we can see that either
$\nu_j(t) \equiv 0$ or $(h_j \circ \varphi)(t) \equiv  0$ on
$(0,\epsilon)$; in particular,
$$\nu_j(t) \frac{d}{ dt }(h_j \circ \varphi)(t)  \equiv 0,
\quad j = 1, \ldots, m.$$
It follows from (b3)-(b4) that
\begin{eqnarray*}
0
&=& \langle \nabla f (\varphi(t)), \frac{d \varphi(t)}{dt} \rangle -
\sum_{i = 1}^{l} \lambda_i(t) \langle \nabla  g_i(\varphi(t)),
\frac{d\varphi(t)}{dt} \rangle - \sum_{j = 1}^{m} \nu_j(t) \langle
\nabla  h_j(\varphi(t)), \frac{d \varphi(t)}{dt} \rangle  \\
&=& \frac{d}{dt}(f \circ \varphi)(t) - \sum_{i = 1}^{l}  \lambda_i(t)
\frac{d}{dt}(g_i \circ \varphi)(t) - \sum_{j = 1}^{m} \nu_j(t)
\frac{d}{dt}(h_j \circ \varphi)(t)\\
&=& \frac{d}{ dt }(f \circ \varphi)(t).
\end{eqnarray*}
Consequently, $f \circ \varphi$ is a constant function on
$(0,\epsilon)$.

Similarly, suppose that the curve $\psi(t)$ lies in $\Sigma(f, S)$ for all $t\in(0,\varepsilon)$. Then $f \circ \psi$ is a constant function on $(0,\epsilon)$.
Since $f$ is continuous, we have $f \circ \varphi
\equiv f \circ \psi \equiv f(0).$
It follows from (b1) that $f$ is constant on $\mathbb{B}_{\epsilon}$, a contradiction.

Therefore, for any $0<\epsilon' < \epsilon, $ there exists $t\in(0,\epsilon')$ such that either $\varphi(t) \in \GAMMA \setminus \Sigma(f, S)$ or
$\psi(t) \in \GAMMA \setminus \Sigma(f, S).$
This, together with (b2), implies (iii).

(iv) This follows from (iii).
\end{proof}

\begin{lemma}\label{lemma3}
Assume that {\rm (LICQ)} holds at $x^*.$ Then there exists $R > 0$ such that for all $x \in \GAMMA \cap \Bbb B_R (x^*) \setminus \{x^*\},$ there exist real numbers $\lambda_i , \nu_j , \mu$ such that
\begin{eqnarray*}
&& \nabla  f(x)  - \sum_{i = 1}^{l} \lambda_i \nabla g_i(x) - \sum_{j
= 1}^{m} \nu_j \nabla h_j(x) - \mu (x - x^*) = 0,\\
&& \nu_j h_j(x) = 0,\ j = 1, \ldots, m.
\end{eqnarray*}
\end{lemma}

\begin{proof}
This follows directly from Lemma~\ref{includex}.
\end{proof}

We will show that in general, in some neighbourhood of $x^*,$ $\GAMMA \setminus \Sigma(f, S)$ is a curve. To see this, it suffices to change the Euclidean norm $\|
\cdot \|$ by a ``generic" one. More precisely, let $\mathscr{P}$ be
the set of symmetric positive definite $n \times n$ matrices. Clearly,
$\mathscr{P}$ is an open semialgebraic subset of $\Bbb
R^{\frac{n(n+1)}{2}}$, where we identify $P = (p_{ij})_{n \times n}
\in \mathscr{P}$ with $(p_{11}, \ldots , p_{1n}, p_{22} , \ldots ,
p_{2n} , \ldots , p_{nn} ) \in \Bbb R^{\frac{n(n+1)}{2}}.$ For each $P
\in \mathscr{P},$ let
\begin{eqnarray*}
\GAMMAP :=  \{x \in S & \ | \ &  \textrm{there exist real
numbers $\kappa, \lambda_i, \nu_j, \mu,$ not all zero, such that } \\
&& \kappa \nabla  f(x)  - \sum_{i = 1}^{l} \lambda_i \nabla g_i(x) -
\sum_{j = 1}^{m} \nu_j \nabla h_j(x) - \mu P(x - x^*) = 0,  \ \textrm{ and } \\
&&  \nu_j h_j(x) = 0,\ j = 1, \ldots, m  \}.
\end{eqnarray*}

\begin{theorem}\label{th::curve}
Assume that {\rm (LICQ)} holds at $x^*$ and the restriction of $f$
on $S$ is nonconstant in some neighbourhood of $x^*.$ Then there exists an open and dense semialgebraic set $\mathscr{U}$ in
$\mathscr{P}$ such that for each ${P} \in \mathscr{U},$ the set $\big(
\GAMMAP\setminus \Sigma(f , S) \big) \cap \Int{\Bbb B_{R_P}(x^*)}$ is
a one-dimensional manifold for some $R_P>0$ depending on $P$.
\end{theorem}

\begin{proof}
Without loss of generality, we assume $x^* = 0.$ Choose a $R>0$ satisfying the conditions in Lemmas~\ref{regular1} and \ref{includex}. For each subset $J := \{j_1,
\ldots, j_k\}$ of $\{1, \ldots, m\},$ let ${\nu} :=
({\nu}_j)_{j \in J} \in \mathbb{R}^{\#J},$ and
\begin{eqnarray*}
X_J :=  \{(x , \kappa , \lambda , {\nu} , \mu) \in \mathbb{R}^n
\times \mathbb{R} \times \mathbb{R}^l \times \mathbb{R}^{\#J}
\times \mathbb{R}  & | &  \kappa^2 + \sum_{i = 1}^l \lambda_i^2 +
\sum_{j \in J} {\nu_j}^2 + \mu^2 = 1, \mu \not= 0, \\
&& 0 < \| x \| <R ,\ h_j(x) > 0 \textrm{ for } j \not \in J \}.
\end{eqnarray*}
Clearly, $X_J$ is a semialgebraic manifold of dimension $n + l + \#J + 1.$
Assume that $X_J \ne \emptyset.$ We define the semialgebraic map
$\Phi_J \colon X_J \times \mathscr{P} \rightarrow \mathbb{R}^n \times
\mathbb{R}^{l} \times \mathbb{R}^{\#J}$ by
\begin{eqnarray*}
\Phi_J(x, \kappa , \lambda, {\nu}, \mu , P) &:= & \big(
\kappa \nabla  f(x) - \sum_{i = 1}^l \lambda_i \nabla g_i(x)
- \sum_{j \in J} {\nu}_j \nabla h_j(x) - \mu P x, \\
&& g_{1}(x), \ldots, g_{l}(x),  h_{j_1}(x), \ldots, h_{j_k}(x) \big).
\end{eqnarray*}
Take any $(x , \kappa , \lambda , {\nu} , \mu , P) \in
\Phi_J^{-1}(0).$ Then $x \not= 0.$ Without loss of generality, we
assume that $x_1 \not= 0.$ Note that $p_{ij} = p_{ji}.$ Then, a direct
computation shows that
\begin{eqnarray*}
\begin{pmatrix}
D_x \Phi_J \ \big |\ D_{(p_{11}, \ldots, p_{1n})} \Phi_J
\end{pmatrix}  =
\left(
\begin{array}{c|c}
$$ \cdots & - \mu
\left(
\begin{array}{ccccc}
	x_1 &  x_2 &  x_3 & \cdots &  x_n\\
	0	    &  x_1 & 0           & \cdots & 0\\
	0           & 0           &  x_1 & \cdots & 0\\
	\vdots      & \vdots      & \vdots      &        & \vdots\\
	0           & 0           & 0           & \cdots & x_1\\
\end{array}\right)$$ \\
\hline
$$ [\nabla g_{1}(x)]^T & 0 $$\\
\vdots & \vdots \\
$$ [\nabla g_{l}(x)]^T & 0$$\\
\hline
$$[\nabla h_{j_1}(x)]^T & 0$$\\
\vdots & \vdots \\
$$[\nabla h_{j_k} (x)]^T & 0$$
\end{array}\right)
,\end{eqnarray*}
where $D_x \Phi_J$ and $D_{(p_{11}, \ldots, p_{1n})} \Phi_J$ denote
the derivative of $\Phi_J$ with respect to $x$ and $(p_{11}, \ldots,
p_{1n}),$ respectively. It follows from Lemma~\ref{regular1} that
the rank of the Jacobian matrix of the map $\Phi_J$ is $n + l + \#J$
and hence $0$ is a regular value of $\Phi_J.$ By the Sard theorem with parameter~\ref{SardTheorem}, there exists a semialgebraic
subset $\Sigma_J \subset \mathscr{P}$ of dimension $< \dim
\mathscr{P}$  such that for each $P \in \mathscr{P} \setminus
\Sigma_J,$ $0$ is a regular value of the map
\begin{align*}
		\Phi_{J, P} \colon X_J \  \rightarrow\  \Bbb R^n \times \Bbb
		R^l \times \Bbb R^{\# J}, \qquad (x, \kappa , \lambda, {\nu},
		\mu)\ \mapsto \  \Phi_{J}(x, \kappa , \lambda, {\nu}, \mu, P).
\end{align*}
Thus, $\Phi_{J , P}^{-1}(0)$ is either empty or an one-dimensional
submanifold of $\Bbb R^n \times \Bbb R \times \Bbb R^l \times \Bbb
R^{\# J} \times \Bbb R.$ By Proposition~\ref{DimensionProposition},
$\dim \pi_J (\Phi_{J , P}^{-1}(0)) \le 1,$ where $\pi_J \colon \Bbb
R^n \times \Bbb R \times \Bbb R^l \times \Bbb R^{\# J} \times \Bbb R
\rightarrow \Bbb R^n$ is the projection on the first $n$ coordinates.
Let $\mathscr{U} := \bigcap\limits_J \big(\mathscr{P} \setminus
\Sigma_J \big),$ where the intersection is taken over all subsets $J$
of $\{1 , \ldots , m\}$ with $X_J \not= \emptyset.$ Taking any $P \in
\mathscr{U},$ by Lemma~\ref{includex}, it is easy to check that
\begin{align*}
\Delta_P := \big( \GAMMAP\setminus \Sigma (f , S) \big)
\bigcap {\Int{\Bbb B_R}}
\subseteq\bigcup_J \pi_J
\big( \Phi_{J , P}^{-1}(0) \big).
\end{align*}
Hence $\dim \Delta_P \le 1.$ On the other hand, by
Lemma~\ref{property} (iv), we can see that $\dim \Delta_P \ge 1.$
Therefore, $\dim \Delta_P = 1.$ By Cell Decomposition Theorem~\ref{CellTheorem}, $\Delta_P$ is a finite disjoint union of
one-dimensional manifolds and points. Now we can choose a $R_P\le R$
to exclude the finitely many points in the union by
$\Int{\mathbb{B}_{R_P}}$. As the remaining one-dimensional manifolds in
$\Int{\mathbb{B}_{R_P}}$ are disjoint, the proof is complete.
\end{proof}

We next show that after a generic linear change of coordinates,
$\GAMMA\setminus\Sigma(f,S)$ is indeed a curve in a neighborhood of
$x^*.$ Let $\mathscr{I}^{n \times n}$ be the set of all invertible
$n\times n$ matrices in $\RR^{n\times n}.$ For $p \in\RR[x]$ and
$A\in\mathscr{I}^{n\times n},$ denote $p^A := p (Ax)$ the polynomial
obtained by applying the change of variables $A$ to $p$ and
\[
S^A := \{x\in{\Bbb R}^n\mid g_i^A(x) = 0,\ i = 1, \ldots, l,\
h_j^A(x)\ge0,\ j=1,\ldots, m\}.
\]

\begin{remark}{\rm
Note that the extremality of $x^*$ as a KKT point of $f$ on $S$
remains the same after an invertible linear change of coordinates,
which means that we can equivalently consider the extremality of $f^A$
at $A^{-1}x^*$ over the set $S^A$ for any invertible matrix
$A\in\RR^{n\times n}$.
}\end{remark}

\begin{corollary}\label{cor::curve2}
Assume that {\rm (LICQ)} holds at $x^*.$ Then there exists a non-empty Zariski open set $\mathcal{E}$ in $\RR^{n\times n}$
such that for each $A\in \mathcal{E},$ the set $\left(\Gamma(f^A,
S^A,A^{-1}x^*)\setminus\Sigma(f^A, S^A)\right)\cap\Int{\Bbb B_{R_A}
(A^{-1}x^*)}$ is a
manifold of dimension one for some $R_A>0$ depending on $A$.
\end{corollary}

\begin{proof}
Since $\nabla f^A(x)=A^T\nabla f(Ax)$ for any $A\in
\mathscr{I}^{n\times n}$, it is easy to check that $\Sigma(f^A, S^A)=A^{-1}(\Sigma(f,
S))$ and $\Gamma(f^A, S^A,A^{-1}x^*)=A^{-1}(\GAMMAA)$.
Let $\mathscr{U}$ be the open and dense semialgebraic set in
$\mathscr{P}$ as described in Theorem~\ref{th::curve}.
Let $\mathscr{U}^{-1} := \{P^{-1}\in\mathscr{P}\mid P\in\mathscr{U}\}$.
As $\mathscr{U}^{-1}$ is also an open and dense semialgebraic set in
$\mathscr{P}$, by \cite[Lemma 1.4]{Ha2017}, there exists a
non-constant polynomial $\mathcal{F} \colon
\mathbb{R}^{\frac{n(n+1)}{2}} \rightarrow \mathbb{R}$ such that
$\mathscr{U}^{-1}\supseteq\{P\in\mathscr{P}\mid\mathcal{F}(P)\neq
0\}$. Let $\mathcal{E}:=\{A\in\mathscr{I}^{n\times n}\mid
\mathcal{F}(AA^T)\neq 0\}$, then $\mathcal{E}$ is a non-empty Zariski
open set in $\RR^{n\times n}$. For each $A\in\mathcal{E}$, we have $A^{-T}A^{-1}\in\mathscr{U}$. By
Theorem~\ref{th::curve}, $\left(\GAMMAA\setminus\Sigma(f,
S)\right)\cap\Int{\Bbb B_{R'_A}(x^*)}$ is a manifold of dimension
one for some $R'_A>0$ depending on $A$. Set $R_A=R'_A/\Vert A\Vert$. Then we can verify that
$\mathbb{B}_{R_A}(A^{-1}x^*)\subseteq A^{-1}(\mathbb{B}_{R'_A}(x^*))$.
Consequently, $\left(\Gamma(f^A, S^A,A^{-1}x^*)\setminus\Sigma(f^A,
S^A)\right)\cap\Int{\Bbb B_{R_A}(A^{-1}x^*)}\subseteq A^{-1}
\left(\left(\GAMMAA\setminus\Sigma(f, S)\right)\cap\Int{\Bbb
	B_{R'_A}(x^*)}\right)$ is a manifold of dimension one.
\end{proof}

In Appendix \ref{sec::correctness}, we will prove that the complex
version of Corollary \ref{cor::curve2} still holds, which is crucial
in the
design of algorithms for testing the extremality of $x^*$.

\section{Faithful radii and types of KKT points}\label{sec::type}

In this section, we first define the so-called {\itshape faithful
radius} of $f$ on $S$ at $x^*$ by means of the tangency variety
$\GAMMA$. Then, we show that the type of $x^*$ can be determined by
the global extrema of $f$ over the intersection of
$\GAMMA$ and the ball centered at $x^*$ with a faithful radius.

\subsection{On faithful radii}

\begin{definition}\label{def::faithr}{\rm
We say that a real number $R > 0$ is a {\em faithful radius} of $f$ on $S$ at
$x^*$ if the following conditions hold:
\begin{enumerate}
\item[(i)] $\Sigma(f, S) \cap \mathbb{B}_R(x^*) = \{ x^* \};$
\item[(ii)] $\GAMMA \cap \mathbb{B}_R(x^*)$ is connected; and
\item[(iii)] $\GAMMA \cap \{x\in\RR^n\mid f(x) = f(x^*) \} \cap \mathbb{B}_R(x^*) = \{ x^* \}.$
\end{enumerate}
}\end{definition}

\begin{theorem}
Assume that {\rm (LICQ)} holds at $x^*.$ The point $x^* \in S$ is an isolated
KKT point of $f$ on $S$ if and only if  there is a faithful radius $R$
of $f$ on $S$ at $x^*.$
\end{theorem}
\begin{proof}
{\em Sufficiency.} This is clear.

{\em Necessity.} Without loss of generality, we assume that $x^* = 0$
and $f(x^*) = 0.$ As $0$ is an isolated KKT point, there exists $R_1>0$ such that $\Sigma(f, S) \cap
\mathbb{B}_{R_1} = \{ 0 \}.$

By Theorem \ref{CellTheorem}, $\GAMMA$ is a disjoint union
of finitely many submanifolds $\Gamma_1,\ldots, \Gamma_s$, each
diffeomorphic to an open hypercube $(0,1)^{\dim(\Gamma_i)}$. Consider
the map $\Phi: x \mapsto \sum_{i=1}^nx_i^2$ on these manifolds. By the
semialgebraic version of Sard's theorem \cite[Theorem
9.6.2]{Bochnak1998}, there are finitely many critical values of the
map $\Phi$ on $\Gamma_1,\ldots,\Gamma_s$. Fix a $R_2\in \Bbb R_+$ to
be the smallest one, then $\GAMMA \cap \mathbb{B}_R$ is connected for
any $0 < R < \sqrt{R_2}.$ To see this, note
that by Proposition~\ref{prop1}, $\GAMMA \cap \mathbb{B}_R$ has
finitely many connected components $\mathcal{C}_1,\ldots,
\mathcal{C}_l$ which are closed in $\Bbb R^n.$ To the contrary,
suppose that $l\ge 2$ and $0\not\in\mathcal{C}_2$. As $\mathcal{C}_2$ is closed and
bounded, the function $\sum_{i=1}^nx_i^2$ reaches its minimum on $\mathcal{C}_2$ at a
minimizer $u$.  Since $\mathcal{C}_2\subseteq \Gamma_i$ for some
$i$, $u$ is a critical point of $\Phi$ on $\Gamma_i$, a contradiction.

Finally, we show that $\GAMMA \cap \{x\in\RR^n\mid f(x)=0\} \cap \Bbb B_R = \{ 0
\}$ for some $R > 0.$ If this is not the case, then by
Lemma~\ref{lemma3} and the Curve Selection
Lemma~\ref{CurveSelectionLemma}, there exist a smooth nonconstant
semialgebraic curve $\varphi(t)$ and semialgebraic functions
$\lambda_i(t), \nu_j(t), \mu(t), t \in (0, \epsilon),$ such
that
\begin{enumerate}
\item[(d1)] $\varphi(t) \in S$ and $f(\varphi(t))=0, $ for $t \in (0, \epsilon);$
\item[(d2)] $\|\varphi(t)\| \rightarrow 0$ as $t \rightarrow 0^+;$
\item[(d3)] $\nabla f(\varphi(t)) - \sum_{i = 1}^{l} \lambda_i(t)
	\nabla g_i(\varphi(t)) - \sum_{j = 1}^{m}\nu_j (t) \nabla
	h_j(\varphi(t)) - \mu(t) \varphi(t) \equiv  0;$ and
\item[(d4)] $\nu_j(t) h_j(\varphi(t)) \equiv  0, j = 1, \ldots, m.$
\end{enumerate}

As is shown in the proof of Lemma~\ref{property} (iii), we may assume that
$$\nu_j(t) \frac{d}{ dt }(h_j \circ \varphi)(t)  \equiv 0,
\quad j = 1, \ldots, m.$$ It follows from (d3) that
\begin{eqnarray*}
0  &=& \langle\nabla f(\varphi(t)), \frac{d\varphi(t)}{dt}
\rangle - \sum_{i = 1}^{l} \lambda_i(t) \langle \nabla  g_i(\varphi(t)), \frac{d
\varphi(t)}{dt} \rangle - \sum_{j = 1}^{m} \nu_j(t)
\langle \nabla  h_j(\varphi(t)), \frac{d \varphi(t)}{dt} \rangle \\
&& -  \mu (t) \langle \varphi(t), \frac{d \varphi(t)}{dt} \rangle \\
&=&  \frac{d}{dt}(f \circ \varphi)(t) -\sum_{i = 1}^{l}
\lambda_i(t) \frac{d}{dt}(g_i \circ \varphi)(t) - \sum_{j = 1}^{m} \nu_j(t)
\frac{d}{dt}(h_j \circ \varphi)(t) - \frac{\mu(t)}{2} \frac{d \|\varphi(t)\|^2}{dt} \\
&=& -\frac{\mu (t)}{2} \frac{d \|\varphi(t)\|^2}{dt}
\end{eqnarray*}
holds for each $t\in(0,\epsilon)$. By the Monotonicity Lemma~\ref{MonotonicityLemma}, there exists $\epsilon' \in (0, \epsilon)$ such that for each $t\in(0,\epsilon')$, it holds that
$\mu(t)=0$. Hence, (d3) implies that $\varphi(t)\in \Sigma(f, S)$ for $t \in (0, \epsilon'),$ a
contradiction.
\end{proof}

Now, we present some sufficient conditions to guarantee a $R>0$ being a
faithful radius, which will be used for computing a faithful radius of
$x^*$ in Section~\ref{sec::computation}.
For a given $\mathscr{R}\in\RR_+$, consider the following condition
\begin{condition}\label{con::curve}
	\begin{enumerate}[\upshape (i)]
	\item $\mathscr{R}$ is an isolation radius of $x^*$ in the sense that
	$\Sigma(f, S) \cap \mathbb{B}_{\mathscr{R}}(x^*)=\{x^*\}.$
	
	\item The vectors $\nabla g_i(x), i = 1, \ldots, l,\ \nabla h_j(x),\ j \in
		J(x),$ and $x-x^*$ are linearly independent for all $x \in S \cap \Bbb B_\mathscr{R}(x^*) \setminus \{x^*\};$
\item For any $u\in \GAMMA \cap {\sf int}(\mathbb{B}_\mathscr{R}(x^*))$ with $u \ne x^*,$ there exist a neighborhood
$\mathcal{O}_u\subset\mathbb{B}_\mathscr{R}(x^*)$
of $u,$ a differentiable map $\phi \colon
(-\varepsilon,\varepsilon)\rightarrow\RR^n$ such that
$\phi((-\varepsilon,\varepsilon))=\GAMMA\cap\mathcal{O}_u,$
$\phi(0)=u$ and $\frac{d\Vert\phi\Vert^2}{d t}(0)\neq 0.$
\end{enumerate}
\end{condition}

\begin{remark}{\rm
 Assume that (LICQ) holds at $x^*$ and $x^*$ is an isolated KKT point, in particular,
Condition~\ref{con::curve}~(i) holds for all $\mathscr{R}> 0$ sufficiently small. In light of Lemma~\ref{includex}, Condition~\ref{con::curve}~(ii) also holds for all $\mathscr{R}> 0$ sufficiently small. By Corollary~\ref{cor::curve2} and the Cell Decomposition Theorem~\ref{CellTheorem}, up to a generic linear change of coordinates, $(\GAMMA\backslash\{x^*\}) \cap {\sf
		int}(\mathbb{B}_R(x^*))$ is a one-dimensional
		smooth manifold for some $R>0$. Then, due to Sard's
		theorem (see, for example, \cite[Corollary~1.1]{Ha2017}), Condition~\ref{con::curve}~(iii) holds for all $\mathscr{R}>0$ small enough. Furthermore, in Section~\ref{sec::computation}, we shall see that a
		$\mathscr{R}>0$ satisfying Condition~\ref{con::curve} can be
		computed by some algebraic computations implemented in the
		current computer algebra systems, like {\scshape Maple}.
}\end{remark}

\begin{theorem}\label{th::main2}
Suppose that $\mathscr{R}\in\RR_+$ satisfies Condition~\ref{con::curve}. Then$,$ any $R\in\RR_+$ with $R<\mathscr{R}$ is a faithful radius of $x^*.$
\end{theorem}

\begin{proof}
Without loss of generality, we assume $x^* = 0.$
We first show that Condition~\ref{con::curve} (iii) implies that
$\GAMMA\cap\mathbb{B}_R$ is connected.  Otherwise, there is a connected component $\mathcal{C}$ such that
$\bfz\not\in\mathcal{C}$.  Since $\GAMMA\cap\mathbb{B}_R$ is closed,
$\mathcal{C}$ is closed by Proposition~\ref{prop1}.  Then, the
function $\Vert x\Vert^2$ reaches its minimum on $\mathcal{C}$ at a
minimizer $u\in\mathcal{C}$.  By the assumption, there exist a
neighborhood $\mathcal{O}_u$ of $u$ and a differentiable mapping $\phi
\colon (-\epsilon,\epsilon)\rightarrow\RR^n$ such that
$\phi((-\epsilon,\epsilon))=\GAMMA\cap\mathcal{O}_u$ and $\phi(0)=u.$
By choosing $\epsilon$ small enough, we may assume that
$\phi((-\epsilon,\epsilon))\subseteq\mathcal{C}\cap\mathcal{O}_u.$
Then, the function $\Vert\phi\Vert^2$ reaches its local minimum at
$0,$ which contradicts Condition \ref{con::curve} (iii).  Hence,
$\GAMMA\cap\mathbb{B}_R$ is connected.

Now assume to the contrary that there exists $\bfz\neq v\in \GAMMA\cap
\{x\in\RR^n\mid f(x) = 0\} \cap \mathbb{B}_R.$ Since $\GAMMA\cap\mathbb{B}_R$ is
connected, there exists a path connecting $\bfz$ and $v$.  Then, $f$
has a local extremum on a relative interior of this path, say $u.$ By
the assumption, there exists a differentiable and semialgebraic
mapping $\phi$ on $(-\epsilon,\epsilon)$ as described in
Condition~\ref{con::curve} (iii).
Then the differentiable function $f \circ \phi$ reaches a
local extremum at $0.$ By the mean value theorem, $$0=\frac{d
(f\circ\phi)}{d t}(0).$$ On the other hand, by Condition~\ref{con::curve} (ii) and (iii), there exist semialgebraic functions
$\lambda_i(t)$, $\nu_j(t)$, $\mu(t)$, $t \in (-\epsilon, \epsilon),$ such
that
\begin{enumerate}
	\item [(e1)] $\nabla f(\phi(t)) - \sum_{i = 1}^{l} \lambda_i(t)
		\nabla g_i(\phi(t)) - \sum_{j = 1}^{m}\nu_j (t) \nabla
		h_j(\phi(t)) - \mu(t) \phi(t) \equiv  0;$
	\item [(e2)] $\nu_j(t) h_j(\phi(t)) \equiv  0, j = 1, \ldots, m.$
\end{enumerate}
As is shown in the proof of Lemma~\ref{property} (iii), we may assume that
$$\nu_j(t) \frac{d}{ dt }(h_j \circ \varphi)(t)  \equiv 0,
\quad j = 1, \ldots, m.$$
It follows from (e1) that
\begin{eqnarray*}
0  &=& \langle\nabla f(\phi(t)), \frac{d\phi(t)}{dt}  \rangle -
\sum_{i = 1}^{l} \lambda_i(t) \langle \nabla  g_i(\phi(t)), \frac{d
\phi(t)}{dt} \rangle - \sum_{j = 1}^{m} \nu_j(t) \langle \nabla
h_j(\phi(t)), \frac{d \phi(t)}{dt} \rangle \\
&&-  \mu (t) \langle \phi(t), \frac{d \phi(t)}{dt} \rangle \\
&=&  \frac{d}{dt}(f \circ \phi)(t) -\sum_{i = 1}^{l}  \lambda_i(t)
\frac{d}{dt}(g_i \circ \phi)(t) - \sum_{j = 1}^{m} \nu_j(t)
\frac{d}{dt}(h_j \circ \phi)(t) - \frac{\mu(t)}{2} \frac{d
\|\phi(t)\|^2}{dt} \\
&=& \frac{d}{dt}(f \circ \phi)(t)-\frac{\mu (t)}{2} \frac{d
\|\phi(t)\|^2}{dt}.
\end{eqnarray*}
Let $t$ tend to $0$, it follows that
\[
0=\frac{d}{dt}(f \circ \phi)(0)=\frac{\mu (0)}{2} \frac{d \|\phi\|^2}{dt}(0).
\]
Since $\phi(0) \not \in \Sigma(f, S)$ by Condition~\ref{con::curve}
(i),  we have $\mu(0)\neq 0$ and so
\[
\frac{d \Vert\phi\Vert^2}{d t}(0)=0,
\]
which contradicts Condition \ref{con::curve} (iii).
Therefore $\GAMMA\cap \{x\in\RR^n\mid f(x) = 0\}
\cap\mathbb{B}_R=\{\bfz\},$ and so $R$ is a faithful radius of
$\bfz$.
\end{proof}

\subsection{On types of isolated KKT points}
For each $R > 0$, let
\begin{equation}\label{eq::fmin}
	\begin{aligned}
		f_R^{\min} &:=& \min\{f(x)\mid x\in S\cap\mathbb{B}_R(x^*)\}, \\
		f_R^{\max} &:=& \max\{f(x)\mid x\in S\cap\mathbb{B}_R(x^*)\}.
	\end{aligned}
\end{equation}

\begin{proposition}\label{coro}
For any $R \in \Bbb R_+,$ we have
\begin{eqnarray*}
f_R^{\min} &=& \min\{f(x)\mid x\in \GAMMA\cap \Bbb{B}_R(x^*)\}, \\
f_R^{\max} &=& \max\{f(x)\mid x\in \GAMMA\cap \Bbb{B}_R(x^*)\}.
\end{eqnarray*}
\end{proposition}
\begin{proof}
This follows immediately from the Fritz-John necessary optimality conditions (see, for example, \cite{Bertsekas2016}).
\end{proof}

\begin{remark}\label{rmk::type}{\rm
Assume that {\rm (LICQ)} holds at $x^*$ and $x^*$ is an isolated KKT
point.  By Corollary~\ref{cor::notconstant}, the following statements hold:
\begin{enumerate}
\item[(i)] If $x^*$ is a local minimizer of $f$ on $S$, then there is a $R > 0$:
\begin{eqnarray*}
f_R^{\max} \ > \ f(x^*) \ = \ f_R^{\min}.
\end{eqnarray*}
\item[(ii)] If $x^*$ is a local maximizer of $f$ on $S$, then there is a $R > 0$:
\begin{eqnarray*}
f_R^{\max} \ = \ f(x^*) \ > \ f_R^{\min}.
\end{eqnarray*}
\item[(iii)]  If $x^*$ is not an extremum point of $f$ on $S$, then for any $R > 0,$
\begin{eqnarray*}
f_R^{\max} \ > \ f(x^*) \ > \ f_R^{\min}.
\end{eqnarray*}
\end{enumerate}}
\end{remark}
Conversely, the next theorem shows that the type of $x^*$ can be determined by
the global extrema of $f$ over the intersection of
$\GAMMA$ and the ball centered at $x^*$ with a faithful radius.
\begin{theorem}\label{th::main}
Assume that {\rm (LICQ)} holds at $x^*.$
Suppose that $R\in \Bbb R_+$ is a faithful radius of $f$ on $S$ at $x^*,$ then the following statements hold$:$
\begin{enumerate}[\upshape (i)]
\item the point $x^*$ is a local minimizer of $f$ on $S$ if and only
	if $f_{R}^{\rm max} > f(x^*) = f_{R}^{\rm min};$
\item the point $x^*$ is a local maximizer of $f$ on $S$ if and only
	if $f_{R}^{\rm max} = f(x^*) > f_{R}^{\rm min};$
\item the point $x^*$ is not an extremum point of $f$ on
	$S$ if and only if $f_{R}^{\max} > f(x^*) > f_{R}^{\min}.$
\end{enumerate}
\end{theorem}

\begin{proof}
By Remark~\ref{rmk::type}, (i) and (ii) are clear if we can prove (iii).

(iii) {\em Necessity.} This is clear by Remark \ref{rmk::type}.

{\em Sufficiency.} By Proposition~\ref{coro}, there exists a point $u \in
\GAMMA\cap\Bbb{B}_R(x^*)$ such that $f(u)=f_R^{\min} <
f(x^*).$ Since $R$ is a faithful radius of $f$ on $S,$ the semialgebraic set $\GAMMA \cap \Bbb B_{R}(x^*)$ is connected, and so, is path connected by
Proposition~\ref{prop1}. Consequently, there exists a continuous and semialgebraic mapping
$\phi \colon [0, 1] \rightarrow \GAMMA \cap \Bbb B_{R}(x^*)$ such that
$\phi(0) = x^*$ and $\phi(1) = u.$ Thanks to the Monotonicity Lemma~\ref{MonotonicityLemma}, we may assume that $\phi(t) \ne x^*$ for all $t \in (0, 1).$ We have
$f(\phi(t)) < f(x^*)$ for all $t \in (0, 1].$ Otherwise, by the
continuity, there exists $\bar{t} \in (0, 1)$ such that
$f(\phi(\bar{t})) = f(x^*).$ Since the radius $R$ is faithful, we have
$\phi(\bar{t}) = x^*$ by the definition, a contradiction.

Similarly, let $f_R^{\max} > f(x^*)$ be reached at $v \in \GAMMA \cap
\Bbb{B}_R(x^*).$ Then there exists a continuous and semialgebraic
mapping $\psi \colon [0, 1] \rightarrow \GAMMA \cap \Bbb B_{R}(x^*)$
such that $x^* \not \in \psi((0, 1)),$ $\psi(0) = x^*,$ $\psi(1) = v$
and $f(\psi(t)) > f(x^*)$ for all $t \in (0, 1].$ Therefore, $x^*$ is not an extremum point of $f.$
\end{proof}

Remark that computing the extrema $f_R^{\min}$ and $f_R^{\max}$ in
$(\ref{eq::fmin})$ is NP-hard (c.f. \cite{RW}). Moreover,
in practice, it is difficult to certify the equalities in Theorem
\ref{th::main} due to numerical errors.
For any $R\in\RR_+$, comparing with Proposition \ref{coro}, define
\begin{eqnarray*}
		f_R^{-} &:=& \min\{f(x)\mid x\in \GAMMA\cap \mathbb{S}_R(x^*)\},\\
		f_R^{+} &:=& \max\{f(x)\mid x\in \GAMMA\cap
			\mathbb{S}_R(x^*)\},
\end{eqnarray*}
where $\mathbb{S}_R(x^*)=\{x\in\RR^n\mid \Vert x-x^*\Vert^2=R^2\}$.
Then, we have the following criterion to determine the type of $x^*$.

\begin{theorem}\label{th::con}
	Suppose that $\mathscr{R}\in\RR_+$ satisfies Condition \ref{con::curve}.
	Then for any $0<R<\mathscr{R},$ it holds that
\begin{enumerate}[\upshape (i)]
	\item the point $x^*$ is a local minimizer if and only if
		$f_R^{-}>f(x^*);$
\item the point $x^*$ is a local maximizer if and only if
	$f_R^{+}<f(x^*);$
\item the point $x^*$ is not an extremum point if and only if
	$f_R^{+}>f(x^*)>f_R^{-}.$
\end{enumerate}
\end{theorem}
\begin{proof}
	By Theorem \ref{th::main2}, $R$ is a faithful radius of $x^*$.
	According to Theorem \ref{th::main} and Definition
	\ref{def::faithr} (iii),
	the ``only if'' parts in (i), (ii) and the ``if'' part in (iii) are clear.

(i). ``if'' part. For any
	$u\in\GAMMA\cap\mathbb{B}_R(x^*)\backslash\{x^*\}$, by
	Condition~\ref{con::curve} (iii), it is easy to see that $u$ is
	path connected with $\mathbb{S}_R(x^*)$. By continuity and the
	definition of faithful radius, $f(u)>f(x^*)$ which implies that
	$f^{\max}_R>f(x^*)=f^{\min}_R$. For details, see \cite[Proposition
	5.1 and Theorem 5.2]{GP2017}.

Similarly, we can prove (ii) and then (iii) follows.
\end{proof}

Consequently, Theorem \ref{th::con} shows that we need not check
equalities to determine the type of $x^*$ as in Theorem \ref{th::main}.
Moreover, computing
$f_R^{-}$ and $f_R^{+}$ can be reduced to solving a zero-dimensional
polynomial system and the inequalities in Theorem \ref{th::con} can be
certified by real root
isolation of the polynomial system. See Section
\ref{sec::implementations} for details.

\section{Computational Aspects}\label{sec::computation}
In this section, according to the sufficient Condition~\ref{con::curve} and Theorem \ref{th::con}, we give an
algorithm to determine the type of an isolated KKT point
$x^*$ of \eqref{eq::P}.
By adding extra variables $z_j, j=1,\ldots,m$, consider the
equality-constrained problem
\begin{equation}\label{eq::P2}
	\left\{\begin{aligned}
		\min_{(x,z)\in \RR^n \times \RR^m}&\ \ f(x)\\
		\text{s.t.}&\ \ g_1(x)=0, \ldots, g_l(x)=0,\\
		&\ \ h_1(x)-z_1^2=0, \ldots, h_m(x)-z_m^2= 0.
	\end{aligned}\right.
\end{equation}
Then it is easy to see that (LICQ) holds at $x^*$ and $x^*$ is an isolated KKT point of $f$ in \eqref{eq::P} if and only if (LICQ) holds at $(x^*, z^*)$ and $(x^*, z^*)$ is an isolated KKT point of $f$ in \eqref{eq::P2}, where $z_j^* = h_j(x^*), j=1,\ldots,m.$ Furthermore, $x^*$ is a local minimizer (resp., maximizer or not extremum point) of \eqref{eq::P}  if and only if $(x^*, z^*)$ is a local minimizer (resp., maximizer or not extremum point) of \eqref{eq::P2}. Hence, without loss of generality, we assume in the following that $S$ is
defined by equalities only, i.e.,
\[
	S:=\{x\in\RR^n\mid g_1(x)=\cdots=g_l(x)=0\}.
\]

\subsection{Algorithm}
Recall that $l\le n-1$ by Remark \ref{rk::licq} (ii).
Let $\mathcal{I}_\Sigma$ be the ideal in $\RR[x]$ generated by the
union of $\{g_1,\ldots,g_l\}$ and the set of maximal minors of
\[
	\begin{aligned}
		&\left[
			\begin{array}{cccc}
				\nabla f(x)& \nabla g_1(x)& \cdots& \nabla g_l(x)
		\end{array}\right].\\
	\end{aligned}
\]
Note that $\Sigma(f,S)\subseteq\bfV_\RR(\mathcal{I}_\Sigma)$ and
$\dim(\mathcal{I}_\Sigma)=0$ is a sufficient condition for the
isolatedness of the KKT point $x^*$.
Similarly, if $l= n-1,$ let $\mathcal{I}_\Gamma :=\langle g_1,\ldots,g_l\rangle$; otherwise, let
$\mathcal{I}_\Gamma$ be the ideal in $\RR[x]$ generated by the union
of $\{g_1,\ldots,g_l\}$ and the set of maximal minors of
\[
	\begin{aligned}
		&\left[
			\begin{array}{ccccc}
				\nabla f(x)& \nabla g_1(x)& \cdots& \nabla g_l(x)&
				x-x^*
		\end{array}\right].\\
	\end{aligned}
\]
Clearly, it holds that $\GAMMA=\bfV_\RR(\mathcal{I}_\Gamma).$
Let $\mathcal{I}_{L}$ be the ideal in $\RR[x]$ generated by the union
of $\{g_1,\ldots,g_l\}$ and the set of maximal minors of
\[
	\begin{aligned}
		&\left[
			\begin{array}{cccc}
				\nabla g_1(x)& \cdots& \nabla g_l(x)&
				x-x^*
		\end{array}\right].\\
	\end{aligned}
\]
By Lemma \ref{includex}, $\bfV_\RR(\mathcal{I}_L)\cap\Bbb B_R
(x^*)=\{x^*\}$ for some $R>0$.
Denote the vanishing ideal $\mathcal{G}:={\bf
I}\left({\overline{\bfV_\CC(\mathcal{I}_\Gamma)\setminus\bfV_\CC(\mathcal{I}_\Sigma)}}^{\mathcal{Z}}\right).$
Theorem \ref{th::dim1} shows that $\dim(\mathcal{G})=1$ up to a generic
linear change of coordinates, which does not change of the type of the
KKT point.

\begin{framed}
\begin{algorithm}\label{al::mainal}{\sf
\noindent Type($f,g_1,\ldots,g_l, x^*$)\\
Input: $f,g_1,\ldots,g_l\in\RR[x]$ with $x^*$ being an isolated KKT
point of $f$ on $S$ defined by $g_i$'s\\
Output: The type of $x^*$ as a KKT point of $f$ over $S$.
\begin{enumerate}[\upshape 1.]
	\item If $\dim(\mathcal{I}_\Gamma)=1$, then let
		$\mathcal{I}=\mathcal{I}_\Gamma$; else if $\dim(\mathcal{G})=1$, then let
		$\mathcal{I}=\mathcal{G}$; otherwise, make a generic linear change of
		coordinates and proceed to step 1;
	\item Compute a $R_1\in\RR_+$ such that
		$\bfV_\RR(\mathcal{I}_\Sigma)\cap \Bbb B_{R_1} (x^*)=\{x^*\}$ and
		$\bfV_\RR(\mathcal{I}_L)\cap \Bbb B_{R_1} (x^*)=\{x^*\}$;
	\item Compute a $R_2\in\RR_+$ such that $\bfV_\RR(\mathcal{I})\cap \Bbb
		B_{R_2} (x^*)\setminus\{x^*\}$ is one-dimensional smooth
		manifold and there is no critical point of the map
\begin{eqnarray*}
	\bfV_\RR(\mathcal{I})\cap \Bbb B_{R_2} (x^*)\setminus\{x^*\} \rightarrow
	\RR, \quad  x \mapsto \|x - x^*\|^2.
\end{eqnarray*}
	\item Fix a positive real number $r<\min\{R_1, R_2\}$. Compare
		$f_r^{-}$ and $f_r^{+}$ with $f(x^*)$, respectively.
	\item If $f_r^{-}>f(x^*)$, return ``local minimizer''; if
		$f_r^{+}<f(x^*)$, return
		``local maximizer''; if $f_r^{-}<f(x^*)<f_r^{+}$, return ``not
		an extremum point''.
\end{enumerate}}
\end{algorithm}
\end{framed}

\begin{theorem}\label{th::al}
Algorithm \ref{al::mainal} runs successfully and is correct. In
particular, any positive
real number $r<\min\{R_1, R_2\}$ is a faithful radius of $x^*$
satisfying Condition \ref{con::curve}.
\end{theorem}
\begin{proof}
	See Appendix \ref{sec::correctness}.
\end{proof}

\subsection{Implementations}\label{sec::implementations}
Now we show
some strategies to implement \textsl{Step 2, 3} and \textsl{4} in
Algorithm \ref{al::mainal}.
We remark that the way to implement each step is not unique, while the
ones are specified below in order to facilitate the complexity
discussions in Section
\ref{sec::complexity}. We use the following subroutines from the
literature.

\begin{enumerate}[-]
	\item {\sf Num}\cite[Algorithm 10.14 and 10.15]{ARG}: For
		univariate polynomials $u,v\in\RR[t]$, {\sf Num}($u,v$)
		returns the number of elements in the set $\{u(t)>0\mid
		t\in\RR, v(t)=0\}$.
	\item {\sf RURr}\cite[Sec. 5.1]{Rouillier1999}: For an ideal
		$I\subset\RR[x]$ with $\dim(I)=0$, {\sf RURr}($I$)
	returns the rational univariate
	representation (RUR) of the points in $\bfV_\RR(I)$, i.e.
	univariate polynomials $v_0,v_i,u_i\in\RR[t]$, $i=1,\ldots,n$,
	such that $x\in\bfV_\RR(I)$ if and only if $v_0(t)=0$,
	$x_i=\frac{u_i(t)}{v_i(t)}$, $i=1,\ldots,n$, for some $t\in\RR$.
\item {\sf AlgSamp}\cite[Algorithm 12.16]{ARG}: For $p\in\RR[x]$ with
	$p(x)\ge 0$ on $\RR^n$ and
	$\bfV_\RR(p)$ bounded, {\sf AlgSamp}($p$) returns the rational univariate
	representation of a set of points which meets every connected
	component of $\bfV_\RR(p)$.
\end{enumerate}

\noindent {\textsl{Step 2}}:
If
$\dim(\mathcal{I}_\Sigma)=\dim(\mathcal{I}_L)=0$, then consider the
ideals
$\widetilde{\mathcal{I}}_\Sigma=\mathcal{I}_\Sigma+\langle\Vert x-x^*\Vert^2-x_{n+1}\rangle$
and
$\widetilde{\mathcal{I}}_L=\mathcal{I}_L+\langle\Vert
x-x^*\Vert^2-x_{n+1}\rangle$. Apply the subroutine {\sf RURr} on
$\widetilde{\mathcal{I}}_\Sigma$ and obtain
$v_0,v_i,u_i\in\RR[t]$, $i=1,\ldots,n+1$. Choose a number
$R_\Sigma>0$ such that {\sf Num}($u_{n+1}v_{n+1},v_0)=${\sf
Num}($u_{n+1}v_{n+1}-R^2_\Sigma v^2_{n+1},v_0$) which holds for all
$R_\Sigma>0$ small enough. In the same way,
choose a number $R_L>0$ for $\widetilde{\mathcal{I}}_L$. Then, we can
set $R_1=\min\{R_\Sigma, R_L\}$.

Otherwise,
$R_1$ may be obtained by computing the distance of each connected component of
$\bfV_\RR(\mathcal{I}_\Sigma)$ and $\bfV_\RR(\mathcal{I}_L)$ to $x^*$
using the critical point method proposed in \cite{AUBRY2002543}.
We refer the readers to \cite[Section 4.3]{GP2017} for the details.

\vskip 5pt

\noindent {\textsl{Step 3:}}
Compute the radical ideal $\sqrt{\mathcal{I}}=\langle
\phi_1,\ldots,\phi_t\rangle$.
Denote $\mathscr{D}$ as the set of the determinants of the Jacobian
matrices $\jac\left(\phi_{i_1},\ldots,\phi_{i_{n-1}},\Vert x-x^*\Vert^2\right)$
for all $\{i_1,\ldots,i_{n-1}\}\subset\{1,\ldots,t\}$. (Note that $t
\ge n - 1$ because $\dim \mathcal{I} = 1.$)
Define
$\Delta_{\mathcal{I}}:=\{\phi_1,\ldots,\phi_t\}\cup\mathscr{D}$ and
\[
	\begin{aligned}
		\mathscr{R}_{\mathcal{I}}&:=\inf\{r\in\RR_+\backslash\{0\}\mid
		\exists x\in\bfV_\RR(\Delta_{\mathcal{I}}), \text{s.t. }
	\Vert x-x^*\Vert^2=r^2\},\\
\end{aligned}
\]
Note that
$\mathscr{R}_{I}$ are set
to be $\infty$ for convenience if no such $x$ in the definition exists.

\begin{proposition}\label{pro::dim}
	$\mathscr{R}_\mathcal{I}>0$ and any number
	$R_2\in(0,\mathscr{R}_\mathcal{I})$ satisfies the condition in
	\textsl{Step 3}.
\end{proposition}
\begin{proof}
Let $\mathcal{I}=\mathcal{I}^{(0)}\cap \mathcal{I}^{(1)}$ be the
equidimensional decomposition of $\mathcal{I}$, where
$\dim(\mathcal{I}^{(i)})=i, i=0,1$.
	Then, $\bfV_\CC(\Delta_{\mathcal{I}})$ contains three parts:
	$\bfV_\CC(\mathcal{I}^{(0)})$, the singular locus of
	$\bfV_\CC(\mathcal{I}^{(1)})$ and the set of critical
	points of the map
		\[
			\bfV_\CC(\mathcal{I}^{(1)}) \rightarrow \CC, \quad  x
			\mapsto \sum_{i=1}^n(x_i- x_i^*)^2.
\]
As the singular locus of $\dim(\mathcal{I}^{(1)})$ is
zero-dimensional and there are finitely many critical values of the
above map by Sard's theorem, we obtain $\mathscr{R}_{\mathcal{I}}>0$.
For the second statement,
see \cite[Lemma 4.10]{GP2017} and \cite[Theorem 4.11]{GP2017} for the
details of the proof.
\end{proof}

Let $q\in\RR[x]$ be the sums of squares of the elements in
$\Delta_{\mathcal{I}}$, then
$\bfV_\RR(\Delta_{\mathcal{I}})=\bfV_\RR(q)$. Apply the subroutine
{\sf  AlgSamp} on $q+(\Vert x-x^*\Vert^2-x_{n+1})^2$ and obtain
$v_0,v_i,u_i\in\RR[t]$, $i=1,\ldots,n+1$. We can set $R_2>0$ to be a
number satisfying {\sf Num}($u_{n+1}v_{n+1},v_0)=${\sf
Num}($u_{n+1}v_{n+1}-R^2_2 v^2_{n+1},v_0$) which holds for all $R_2>0$
small enough by the proof of Proposition \ref{pro::dim}.

\vskip 5pt
\noindent {\textsl{Step 4:}} As $\dim(\mathcal{I})=1$, the ideal
$\widetilde{\mathcal{I}}=\mathcal{I}+\langle\Vert
x-x^*\Vert^2-r^2\rangle$
is zero-dimensional for any $0<r<\mathscr{R}_{\mathcal{I}}$ (c.f.
\cite[Proposition 5.3]{GP2017}).  Apply the subroutine
{\sf  RURr} on
$\widetilde{\mathcal{I}}+\langle
f(x)-x_{n+1}\rangle\subset\RR[x,x_{n+1}]$ and obtain
$v_0,v_i,u_i\in\RR[t]$, $i=1,\ldots,n+1$. Compute the numbers
$n_1=${\sf
Num}($u_{n+1}v_{n+1}-f(x^*)v^2_{n+1},v_0$) and  $n_2=${\sf
Num}($f(x^*)v^2_{n+1}-u_{n+1}v_{n+1},v_0$). If $n_1=0$, then
$f_r^+<f(x^*)$. If $n_2=0$, then $f_r^->f(x^*)$. If $n_1>0$ and
$n_2>0$, then $f_r^-<f(x^*)<f_r^+$.

\begin{remark}{\rm
		(i) For a zero-dimensional ideal $I\subset\RR[x]$, the command
		{\sf Isolate} in {\scshape Maple}, which can employ the RUR
		algorithm \cite{Rouillier1999} as a subroutine, is available to compute
		isolating intevals for each point in $\bfV_\RR(I)$. Hence, in
		practice, we can use this command in \textsl{Step 2} and \textsl{Step 4}, as well as in \textsl{Step 3} if $\dim(\langle
		\Delta_{\mathcal{I}}\rangle)=0$ (see Example \ref{ex::ex2}).
		However, the whole arithmetic complexity of the
		command {\sf Isolate} is not known to us.
		
(ii).  Although Algorithm \ref{al::mainal} works in
theory for any problem of the form \eqref{eq::P} with an
isolated  KKT point $x^*\in\RR^n$ and $f$, $g_i$'s, $h_j$'s
$\in\mathbb{R}[X],$ we would like to remark that if either some
coordinate of $x^*$ is not rational or some of $f$, $g_i$'s, $h_j$'s
are not in $\mathbb{Q}[X],$ then the defining polynomials of
$\Gamma(f, S, x^*)$ may not be in $\mathbb{Q}[X]$ and some
difficulties may arise in implementing the algorithms proposed in this
paper. That is because many algebraic computations which can be done
in the current computer algebra systems, like {\scshape Maple}, is
more efficient or only available in the rational number field.
}
\end{remark}

\begin{example}\label{ex::ex2}{\rm
		We consider three optimization problems in the following. For
		each one,
it is easy to check that $\bfz$ is a KKT point with the second-order
necessary optimality condition holds but the second-order sufficient
condition does not. Hence, we can not decide the type of $\bfz$ by
linear algebra.

For each problem, we can check in {\scshape Maple}
that $\dim(\mathcal{I}_\Sigma)=\dim(\mathcal{I}_L)=\dim(\langle
		\Delta_{\mathcal{I}}\rangle)=0$, $\dim(\mathcal{I}_\Gamma)=1$
		and $\mathcal{I}_\Gamma$ itself is radical. Hence, $\mathcal{I}=\mathcal{I}_\Gamma$
		without linear change of coordinates. The command
		{\sf Isolate} in {\scshape Maple} enables us to compute $R_1$
		in \textsl{Step 2} and $R_2$ in \textsl{Step
		3}, as well as intervals $[a_1,b_1]$ and $[a_2,b_2]$ such that
		$f_r^{-}\in[a_1,b_1]$, $f_r^{+}\in[a_2,b_2]$ and
		$f(0)\not\in[a_1,b_1]\cup[a_2,b_2]$ in \textsl{Step 4}.
		Then, by Theorem~\ref{th::con}, it is easy to see that $0$ is
		a local minimizer if
		$a_1>f(0)$, $0$ is a local maximizer if $b_2<f(0)$ and
		$0$ is not an extremum point if $a_1<f(0)<b_2$.
The whole process for each problem takes only a few seconds on a
laptop with two 1.3 GHz cores and 8GB RAM.
\begin{enumerate}[\upshape (1)]
	\item
		Consider the optimization problem
		\begin{equation}\label{eq::ex2}
			\min_{x\in\RR^3}\ 2x_2^4+x_3^4-4x_1^2\qquad \text{s.t.}\quad
			-x_2x_3-x_3^2+2x_1=0.
		\end{equation}
		As $\varepsilon\rightarrow 0$, the two sequence of feasible points
		$(\varepsilon^2,\varepsilon,\varepsilon)$ and $(0,\varepsilon,0)$
		imply that $\bfz$ is not an extremum point of \eqref{eq::ex2}.

		We get that $R_1$ can be chosen to be any positive number and
		$R_2=\sqrt{\frac{30592520018291640355}{1152921504606846976}}\approx
		5.1511$ in \textsl{Step} 2 and 3.
		Let $r=1$ in \textsl{Step} 4, we obtain
		$a_1=-\frac{8133982313870021995}{18446744073709551616}\approx
		-0.4409$
		and $b_2=2$, which certify that $\bfz$ is not an extremum point of
		\eqref{eq::ex2}.
	\item
Consider the optimization problem
\begin{equation}\label{eq::ex3}
	\min_{x\in\RR^3}\ x_1^2+x_2^2+x_3^3\qquad \text{s.t.}\quad
	-x_1x_3^2+x_2x_3^2+x_2^2+x_1=0
\end{equation}
As $\varepsilon\rightarrow 0$, the two sequence of feasible points
$(0,0,\varepsilon)$ and $(0,0,-\varepsilon)$
imply that $\bfz$ is not an extremum point of \eqref{eq::ex3}.

We get
$R_1=\sqrt{\frac{7456077067994313975}{2305843009213693952}}\approx
1.7982$ and
$R_2=\sqrt{\frac{32794211686594305343}{73786976294838206464}}\approx
0.6667$ in \textsc{Step} 2 and 3.
Let $r=\frac{1}{2}$ in \textsl{Step} 4, we obtain
$a_1=-\frac{1}{8}$ and
$b_2=\frac{9223372036854837203}{36893488147419103232}\approx 0.2500$, which
certify that $\bfz$  is not an extremum point of \eqref{eq::ex3}.
\item 		
Consider the optimization problem
\begin{equation}\label{eq::ex}
	\left\{
	\begin{aligned}
		\min_{x\in\RR^3}&\
		-x_1x_2^4-x_3x_2^4+x_1^4+2x_1^2x_3^2+x_3^4-4x_1^2x_3-4x_3^3+x_1^2+4x_3^2\\
	\text{s.t.}&\ x_1^2+x_2^2+x_3^2-2x_3=0.
	\end{aligned}\right.
\end{equation}
Note that the objective can be rewritten as
$x_1^2+(-x_1^2-x_3^2+2x_3)^2-x_1x_2^4-x_3x_2^4$. Then, the value of
the objective at any feasible point $u\in\RR^3$ with $\Vert
u\Vert^2<\frac{1}{4}$ is $u_1^2+(1-u_1-u_3)u_2^4>0$,
which implies that $\bfz$ is a strict local minimizer of
\eqref{eq::ex}.

We get $R_1=\sqrt{\frac{32743693424004235509}{36893488147419103232}}\approx
0.9420$ and $R_2$ can be chosen to be any positive number less than
$2$ in \textsc{Step} 2 and 3.
Let $r=\frac{4}{5}$ in \textsc{Step} 4, we obtain
$a_1=\frac{5847027727233270915}{36893488147419103232}\approx 0.1585$
and $b_2=\frac{19833939228052796449}{36893488147419103232}\approx 0.5376$, which
certify that $\bfz$ is a minimizer.
\end{enumerate}
As mentioned in Section \ref{sec::intro}, the type of the KKT point
$x^*$ can be determined as a quantifier elimination problem by CAD
based algorithms. We apply the {\sf QuantifierElimination}
command (c.f. \cite{chenICMSa,CHEN201674}) of the
{\sf RegularChains}\footnote{RegularChains: http://www.regularchains.org/}
library in {\scshape Maple} to the above problems by determining the
truth of the sentences \eqref{eq::cad}. For the problem \eqref{eq::ex2}, it took
about 17 hours to determine the truth of \eqref{eq::cad}. For the
problems \eqref{eq::ex3} and \eqref{eq::ex}, the {\sf QuantifierElimination}
command kept running for days without any output, which shows
the efficiency of our method.
}
\end{example}

\section{Conclusions} \label{SectionConclusions}
By investigating some properties of the set of KKT points in
\eqref{eq::P} and the tangency variety of $f$ at the isolated KKT
point $x^*$ over $S$, we give the definition of faithful radius
of $x^*$ and show that the type of $x^*$ can be determined by the global
extrema of $f$ over the intersection of $S$ and the Euclidean ball
centered at $x^*$ with a faithful radius. An algorithm involving
algebraic computations for determining the type of $x^*$ is presented.

If $x^*$ is a non-isolated KKT point, then the method proposed in this
paper does not apply. In particular, since the condition (iii) in
Definition \ref{def::faithr} does not hold for any $R$, we can not
determine the local extremality of $x^*$ by investigating the local
values of $f$ on its tangency variety on $S$ at $x^*$ as in Theorem
\ref{th::main} and \ref{th::con}.
The extension of our method in the present paper to
the non-isolated case will be studied in future work.

\subsection*{Acknowledgments}
The first author was supported by the Chinese National Natural Science Foundation under grant 11571350.
The second author was supported by the National Research Foundation of Korea (NRF) Grant funded by the Korean Government (NRF-2019R1A2C1008672).
The third author was supported by Jiangsu Planned Projects for Postdoctoral Research Funds 2019 (no. 2019K151).
The forth author was supported by Vietnam National Foundation for Science and Technology Development (NAFOSTED), grant 101.04-2019.302

\appendix
\section{On Algorithm \ref{al::mainal}}\label{appendix}

\subsection{Correctness}\label{sec::correctness}

For a given invertible matrix
$A\in\mathscr{I}^{n\times n},$ replace $f$, $g_i$'s, $x^*$ by $f^A$,
$g^A_i$'s, $A^{-1}x^*$ in
the definition of $\mathcal{I}_\Gamma, \mathcal{I}_\Sigma$ and denote
the resulting ideals by $\mathcal{I}^A_\Gamma$,
$\mathcal{I}^A_\Sigma$, respectively.

\begin{theorem}\label{th::dim1}
There exists a non-empty Zariski open set $\mathcal{E} \subset \mathbb{C}^{n \times
n}$ such that for all $A\in\mathcal{E}\cap\RR^{n\times n},$ the
Zariski closure
${\overline{\bfV_\CC(\mathcal{I}^A_\Gamma)\setminus\bfV_\CC(\mathcal{I}^A_\Sigma)}}^\mathcal{Z}$
is a one-dimensional algebraic variety in $\CC^n.$
\end{theorem}
\begin{proof}
	Since $\nabla f^A(x)=A^T\nabla f(Ax)$ for any $f(x)\in\RR[x]$ and
	$A\in\mathscr{I}^{n\times n}$, we have
	$\bfV_\CC(\mathcal{I}_\Sigma^A)=A^{-1}(\bfV_\CC(\mathcal{I}_\Sigma))$.
Let $S_\CC=\bfV_\CC(\langle g_1,\ldots,g_l\rangle)$.
Denote $\mathscr{S}_\CC^{n\times n}$ as the set of symmetric matrices
in $\CC^{n\times n}$, which can be identified with the space $\CC^{\frac{n(n+1)}{2}}$.
	For any
	$P\in\mathscr{S}_\CC^{n\times n}$, define
	\[
		\begin{aligned}
			\Gamma_{\CC,P}(f,S_\CC,x^*) := \{x\in S_\CC\mid&\ \exists \kappa, \lambda_i,
				\mu\in\CC\ \text{not all zeros, s.t. } \\
				&\ \kappa\nabla f(x)-\Sigma_{i=1}^l\lambda_i\nabla g_i(x)-\mu
				P(x-x^*)=0\}.
			\end{aligned}
	\]
	Then, it is easy to check that
	$\bfV_\CC(\mathcal{I}^A_\Gamma)=A^{-1}(\Gamma_{\CC,A^{-T}A^{-1}}(f,S_\CC,x^*))$
	and hence
	$\bfV_\CC(\mathcal{I}^A_\Gamma)\setminus\bfV_\CC(\mathcal{I}^A_\Sigma)
	=A^{-1}(\Gamma_{\CC,A^{-T}A^{-1}}(f,S_\CC,x^*)\setminus\bfV_\CC(\mathcal{I}_\Sigma))$.

Define the map
$\Phi\colon
\left(\CC^n\setminus\bfV_\CC(\mathcal{I}_\Sigma)\right)\times\CC\times\CC^l\times\mathscr{S}_\CC^{n\times
n}\rightarrow\CC^n\times\CC^l$
by
\[
\Phi(x,\kappa,\lambda, P):=\big(\kappa\nabla f(x)-\sum_{i =
1}^l\lambda_i \nabla g_i(x)-P(x-x^*), g_1(x), \ldots, g_l(x)\big).
\]
Here, we rescale the coefficient $\mu$ to be $1$ since $x$ in the
domain of $\Phi$ is from $\CC^n\setminus\bfV_\CC(\mathcal{I}_\Sigma).$ Similar to the proof of Theorem \ref{th::curve},
it can be shown that $0$ is a regular value of $\Phi.$ Then according
to the algebraic version of Thom's weak transversality theorem
(c.f. \cite[Ch.~3, Theorem~3.7.4]{BC2000}, \cite{Bank2010}, \cite[Proposition~B.3]{SS2013}),
there exists a Zariski closed subset
$\mathcal{V}_\CC\subset\mathscr{S}_\CC^{n\times n}$ such that for all
$P\in\mathscr{S}_\CC^{n\times n}\backslash\mathcal{V}_\CC$,
$0$ is a regular value of the map
\[
\Phi_P \colon \left(\CC^n\setminus\bfV_\CC(\mathcal{I}_\Sigma)\right)\times\CC\times\CC^l
\rightarrow\CC^n\times\CC^l\qquad(x, \kappa , \lambda)\ \mapsto \  \Phi(x, \kappa, \lambda, P).
	\]
It follows that $\Phi^{-1}_P(0)$ is either empty or a one-dimensional
quasi-affine set of $\CC^{n+1+l}$.
Note that
$\Gamma_{\CC,P}(f,S_\CC,x^*)\setminus\bfV_\CC(\mathcal{I}_\Sigma)$ is
the projection of $\Phi^{-1}_P(0)$ on the first $n$ coordinates.
Then we have
$\dim(\overline{\Gamma_{\CC,P}(f,S_\CC,x^*)
\setminus\bfV_\CC(\mathcal{I}_\Sigma)}^\mathcal{Z})\le 1$ for all
$P\in\mathscr{S}_\CC^{n\times n}\backslash\mathcal{V}_\CC$.
As the Zariski closure
$$\mathcal{V}_\CC^{-1} := \overline{\{P\in\mathscr{S}_\CC^{n \times n}
\ | \ P^{-1}\in\mathcal{V}_\CC\}}^{\mathcal{Z}}$$
is an algebraic set in $\mathbb{C}^{\frac{n(n+1)}{2}}$,
the set $$\{A\in\mathcal{\CC}^{{n \times n}} \ | \
AA^T\in\mathcal{V}_\CC^{-1}\}$$
is an algebraic set in $\CC^{n\times n}$. Let $\mathscr{I}_\CC^{n\times
n}$ be the set of invertible matrices in $\CC^{n\times n}$. It follows that
$\mathcal{E} := \{A \in \mathcal{\CC}^{{n \times n}} \ | \
AA^T\not\in\mathcal{V}_\CC^{-1}\}\cap\mathscr{I}^{n\times n}_\CC$
is an non-empty Zariski open set of $\mathcal{\CC}^{n \times n}.$
Then, for all $A \in \mathcal{E}\cap\RR^{n\times n}$, ${\overline{\bfV_\CC(\mathcal{I}^A_\Gamma)\setminus\bfV_\CC(\mathcal{I}^A_\Sigma)}}^\mathcal{Z}
=\overline{A^{-1}(\Gamma_{\CC,A^{-T}A^{-1}}(f,S_\CC,x^*)\setminus\bfV_\CC(\mathcal{I}_\Sigma))}^\mathcal{Z}$
is a one-dimensional algebraic variety in $\CC^n$.
\end{proof}

\begin{proposition}\label{prop::r}
Suppose that (LICQ) holds at $x^*$ and $x^*$ is an
isolated KKT point. Then, there exists a $R_1\in\RR_+$ such that
$\bfV_\RR(\mathcal{I}_\Sigma)\cap \Bbb B_{R_1} (x^*)=\{x^*\}$ and
$\bfV_\RR(\mathcal{I}_L)\cap \Bbb B_{R_1} (x^*)=\{x^*\}$. For the
ideal $\mathcal{I}$ in Algorithm \ref{al::mainal}, it holds that
$\GAMMA\cap\mathbb{B}_R(x^*)=\bfV_\RR(\mathcal{I})\cap\mathbb{B}_R(x^*)$
any $0<R<R_1$.
\end{proposition}
\begin{proof}
	By Lemma \ref{includex}, there exists a $\widetilde{R}\in\RR_+$ such that
	$\bfV_\RR(\mathcal{I}_L)\cap \Bbb B_{\widetilde{R}}
	(x^*)=\{x^*\}$. It is obvious that
	$\Sigma(f,S)\subseteq\bfV_\RR(\mathcal{I}_\Sigma)$. Due to Lemma
	\ref{regular1}, there exists a $\widehat{R}\in\RR_+$ such that
	$\Sigma(f,S)\cap \Bbb B_{\widehat{R}}
	(x^*)=\bfV_\RR(\mathcal{I}_\Sigma)\cap  \Bbb B_{\widehat{R}}
	(x^*)$. As $x^*$ is an isolated KKT point, by shrinking
	$\widehat{R}$ if necessary, we have
	$\bfV_\RR(\mathcal{I}_\Sigma)\cap \Bbb B_{\widehat{R}}
	(x^*)=\{x^*\}$. Consequently, we can let
	$R_1=\min\{\widetilde{R},\widehat{R}\}$.

As $\GAMMA=\bfV_\RR(\mathcal{I}_\Gamma)$, it is clear that
$\GAMMA\cap\mathbb{B}_R(x^*)=\bfV_\RR(\mathcal{I})\cap\mathbb{B}_R(x^*)$ for any
$R\in\RR_+$ if $\mathcal{I}=\mathcal{I}_{\Gamma}$. Thus, we consider the
case when $\mathcal{I}=\mathcal{G}$. As $R<R_1$, we have
	$\Sigma(f,S)\cap\mathbb{B}_R(x^*)=\bfV_\RR(\mathcal{I}_\Sigma)\cap\mathbb{B}_R(x^*)
	= \{x^*\}$. Then, it is easy to see that
	\[
		\begin{aligned}
			\GAMMA\cap\mathbb{B}_R(x^*)&=\overline{\GAMMA\backslash\Sigma(f,S)}\cap\mathbb{B}_R(x^*)\\
			&\subseteq\overline{\bfV_\RR(\mathcal{I}_\Gamma)\backslash\bfV_\RR(\mathcal{I}_\Sigma)}^{\mathcal{Z}}
			\cap\mathbb{B}_R(x^*)\subseteq\bfV_\RR(\mathcal{G})\cap\mathbb{B}_R(x^*).
		\end{aligned}
\]
It is clear that
$\GAMMA\cap\mathbb{B}_R(x^*)\supseteq\bfV_\RR(\mathcal{G})\cap\mathbb{B}_R(x^*)$
and thus the conclusion follows.
\end{proof}

\begin{proof}[Proof of Theorem \ref{th::al}]
As $\dim(\mathcal{I})=1$, by Sard's theorem (see, for example,
\cite[Corollary~1.1]{Ha2017}),
a real number $R_2>0$ in \textsl{Step 3} of Algorithm \ref{al::mainal}
always exists. Then, due to Theorem \ref{th::dim1} and Proposition
\ref{prop::r},  Algorithm \ref{al::mainal} can run successfully.
	It is clear that the number $R_1$ satisfies Condition
	\ref{con::curve} (i) and (ii). By the second statement of
	Propostion \ref{prop::r}, Condition \ref{con::curve} (iii) holds
	for any $0<r<\min\{R_1,R_2\}$ which implies that $r$ is a faithful
	radius. Then, the correctness of Algorithm \ref{al::mainal}
	follows by Theorem \ref{th::con}.
\end{proof}

\subsection{Discussions on complexity}\label{sec::complexity}

As the implementations involve algebraic
computations of vanishing ideals, critical point method, the radical of
an ideal and so on,
the complexity in our algorithms depends heavily on these
corresponding algorithms. Hence, we leave the complete complexity
analysis of Algorithm \ref{al::mainal} as our future work. Instead, we
present a general discussion on the complexity under the assumption
\begin{assumption}\label{ass::complexity}
	$\dim(\mathcal{I}_\Sigma)=\dim(\mathcal{I}_L)=0$ and the ideal
	$\mathcal{I}_{\Gamma}$ is radical.
\end{assumption}

We first recall the arithmetic complexity of the following
subroutines from the literature.

\begin{enumerate}[-]
	\item {\sf Num}\cite[Algorithm 10.14 and 10.15]{ARG}:
		$O(\deg(v)^2\deg(u)+\deg(v)^4\log_2(\deg(v)))$
	\item {\sf RURr}\cite[Sec. 5.1]{Rouillier1999}:
	$d^{O(n)}$ where $d$ is the maximal degree of the generators of
	$I$
\item {\sf AlgSamp}\cite[Algorithm 12.16]{ARG}: $\deg(p)^{O(n)}$
\end{enumerate}

Denote by $D$ the maximal degree of $f,g_1,\ldots,g_l$. For
simplicity, we use
$Dl+D$ to bound
the degrees of the generators in $\mathcal{I}_\Gamma$,
$\mathcal{I}_\Sigma$ and $\mathcal{I}_L$.
Note that $n\ge 2$, otherwise the problem is trivial.

\vskip 8pt

\noindent{\textsl{Step 1}}: As $\dim(\mathcal{I}_\Sigma)=0$, we have
$\mathcal{I}=\mathcal{I}_\Gamma$ without computing the vanishing ideal
$\mathcal{G}$.

\vskip 5pt

\noindent {\textsl{Step 2}}:
As
$\dim(\mathcal{I}_\Sigma)=\dim(\mathcal{I}_L)=0$, the arithmetic
complexity of applying {\sf RURr} on $\widetilde{\mathcal{I}}_\Sigma$
and $\widetilde{\mathcal{I}}_L$ is $(Dl+D)^{O(n+1)}$. As the numbers of
the points in $\bfV_\RR(\mathcal{I}_\Sigma)$ and
$\bfV_\RR(\mathcal{I}_L)$ are both bounded by the B{\'e}zout number
$(Dl+D)^n$, the degrees
of $v_0,u_i$,$v_i$ returned by {\sf RURr} are bounded by
$(Dl+D)^n$ \cite{Rouillier1999}. Therefore, the arithmetic complexity
of applying the
subroutine {\sf Num} is $O(n(Dl+D)^{4n}\log_2(Dl+D))$.

\vskip 5pt

\noindent {\textsl{Step 3:}}
Recall that $\mathcal{I}=\mathcal{I}_\Gamma$ which is assumed to be
radical.
Clearly, the degrees of polynomials in $\Delta_{\mathcal{I}}$ is
bounded by $n(Dl+D)$.
As $2n(Dl+D)\ge 4$, the arithmetic complexity of applying the
subroutine {\sf AlgSamp}
is $(2n(Dl+D))^{O(n+1)}$ and the degrees
of $v_0,u_i$,$v_i$ returned by {\sf AlgSamp} are bounded by
${O(2n(Dl+D))}^{n+1}$ \cite[pp. 493]{ARG}.
Therefore, the arithmetic complexity
of applying the
subroutine {\sf Num} is
$O({(2n(Dl+D))}^{4(n+1)}(n+1)\log_2(2n(Dl+D)))$.

\vskip 5pt

\noindent {\textsl{Step 4:}} As the maximal degree of the generators
of $\widetilde{\mathcal{I}}+\langle
f(x)-x_{n+1}\rangle$ is bounded by $Dl+D$, the arithmetic
complexity of applying {\sf RURr} on it is $(Dl+D)^{O(n+1)}$.
As the number of
the points in $\bfV_\RR(\widetilde{\mathcal{I}})$ is both bounded by
the B{\'e}zout number $(Dl+D)^n$, the degrees
of $v_0,u_i$,$v_i$ returned by {\sf RURr} are bounded by
$(Dl+D)^n$ \cite{Rouillier1999}. Therefore, the arithmetic complexity
of applying the
subroutine {\sf Num} is $O(n(Dl+D)^{4n}\log_2(Dl+D))$.

\begin{remark}{\rm
		Recal that the arithmetic complexity for solving the
		quantifier elimination problems \eqref{eq::cad} by the CAD is
		${((l+3)D)^{O(1)}}^{n+1}$ if $m=0$ \cite[Excercise 11.7]{ARG},
		which is doubly exponential in $n$.
		Comparitively, under Assumption \ref{ass::complexity}, the
		method proposed in this paper has a lower complexity, which
		can be observed from the numerical experiments in Section
		\ref{sec::computation}.
}
\end{remark}

\end{document}